\documentclass[10pt]{amsart}
\usepackage[dvips]{epsfig}
\usepackage{amscd}
\usepackage{amssymb}
\usepackage{amsthm}
\usepackage{amsmath}
\usepackage{latexsym}
\usepackage[colorlinks]{hyperref}

\newcommand{\eps}{\varepsilon}

\newcommand{\D}{\mathcal{D}}

\newcommand{\Grad}{\mathrm{\nabla}}

\newcommand{\R}{\mathbf{R}}

\newcommand{\dx}{\, dx}

\newcommand{\dt}{\, dt}

\newcommand{\sgn}[1]{\mathrm{sign}^{+}\left(#1\right)}

\newcommand{\sgnm}[1]{\mathrm{sign}^{-}\left(#1\right)}

\newcommand{\Const}[1]{\mathrm{Const}_{#1}}

\newcommand{\pt}{\ensuremath{\partial_t}}

\newcommand{\Do}{\ensuremath{Q_T}}
\newcommand{\Om}{\ensuremath{\Omega}}

\newcommand{\N}{{\mathbb N}}

\newcommand\bel{\begin{equation}\label}
\newcommand\ee{\end{equation}}
\newcommand\ba{\begin{array}}
\newcommand\bal{\begin{array}{l}}
\newcommand\ea{\end{array}}
\def\dsp{\displaystyle}

\newtheorem{theo}{Theorem}[section]

\newtheorem{lem}{Lemma}[section]
\newtheorem{rem}{Remark}[section]
\newtheorem{defi}{Definition}[section]

\def\ptl{{\partial}}
\newcommand{\sign}{{\rm sign}}

\newcommand{\grad}{{\nabla}}
\def\dist{{\rm dist\,}}

\newcommand\1{\hbox{\hbox{1}\kern-.3em I}}

\def\dsp{\displaystyle}
\def\div{{\rm div\,}}

\def\grad{{\,\nabla }}
\def\meas{{\rm meas\,}}
\def\dist{{\rm dist\,}}

\def\al{{\alpha}}
\def\eps{{\varepsilon}}
\def\Om{{\Omega}}
\def\ptl{{\partial}}

\def\ba{\begin{array}}
\def\bal{\begin{array}{l}}
\def\ea{\end{array}}
\def\be{\begin{equation}}
\def\bel{\begin{equation}\label}
\def\ee{\end{equation}}
\def\iint{\int\!\!\int}

\def\sign{{\rm sign\,}}

\def\D{{D}}

\def\pt{\ensuremath{\partial_t}}

\def\D{{\mathcal D}}

\def\Del{{\scriptstyle \Delta}}

\def\char{{1\!\mbox{\rm l}}}

{%

\begin{enumerate}}%
{\end{enumerate}
}
\newenvironment{Definitions}
{%

\begin{enumerate}}%
{\end{enumerate}
}
\newenvironment{DefinitionsP}
{%

\begin{enumerate}}%
{\end{enumerate} }

\begin{document}

\title[Well-posedness for degenerate parabolic equations] {Well-posedness
results for triply nonlinear degenerate parabolic equations}

\date{\today}

\author[B. Andreianov]{B. Andreianov}
\address[Boris Andreianov]{\newline
    Laboratoire de Math\'ematiques\newline
    Universit\'e de Franche-Comt\'e\newline
    16 route de Gray\newline
    25 030 Besanc$\hspace*{-3pt}_{_{^{\mathsf ,}}}$on
 Cedex, France}
\email[]{boris.andreianov\@@univ-fcomte.fr}

\author[M. Bendahmane]{M. Bendahmane}
\address[Mostafa Bendahmane]{\newline
     LAMFA, UMR-CNRS 6140\newline
    Universit\'e de Picardie Jules Verne,\newline
    33 rue Saint Leu, 80038 Amiens, France}
\email[]{ mostafa\_bendahmane\@@yahoo.fr}

\author[K. H. Karlsen]{K. H. Karlsen}
\address[Kenneth H. Karlsen]{\newline
    Centre of Mathematics for Applications \newline
    University of Oslo\newline
    P.O. Box 1053, Blindern\newline
    N--0316 Oslo, Norway}
\email[]{kennethk@math.uio.no}
\urladdr{http://folk.uio.no/kennethk}

\author[S. Ouaro]{S. Ouaro}
\address[Stanislas Ouaro]{\newline
    Laboratoire d'Analyse Math\'ematique des \'Equations LAME\newline
    UFR Sciences Exactes et Appliqu\'ees, University of Ouagadougou\newline
    03 BP 7021 Ouaga 03\newline
    Ouagadougou, Burkina-Faso}
\email[]{souaro@univ-ouaga.bf}



\thanks{The work of K. H. Karlsen was supported by the Research Council of Norway through
an Outstanding Young Investigators Award. The work of S. Ouaro was
supported by the fundings from SARIMA, and then from the AUF. A
part of this work was done while B. Andreianov and S. Ouaro
enjoyed the hospitality of the Centre of Mathematics for
Applications (CMA) at the University of Oslo, Norway. This article
was written as part of the international research program on
Nonlinear Partial Differential Equations at the Centre for
Advanced Study at the Norwegian Academy of Science and
 Letters in Oslo during the academic year 2008--09.}

\begin{abstract}
We study the well-posedness of triply nonlinear degenerate
elliptic-parabolic-hyperbolic problem
$$
b(u)_t - \div \tilde{\mathfrak a}(u,\grad\phi(u))+\psi(u)=f,
\quad u|_{t=0}=u_0
$$
in a bounded domain with homogeneous Dirichlet boundary
conditions.
The nonlinearities $b,\phi$ and $\psi$ are supposed to
be continuous  non-decreasing, and the nonlinearity
$\tilde{\mathfrak a}$ falls within the Leray-Lions framework. Some
restrictions are imposed on the dependence of $\tilde{\mathfrak
a}(u,\grad\phi(u))$ on $u$ and also on the set where $\phi$
degenerates.
A model case is $\tilde{\mathfrak a}(u,\grad\phi(u))
=\tilde{\mathfrak{f}}(b(u),\psi(u),\phi(u))+k(u)\mathfrak{a}_0(\grad\phi(u)),$
with $\phi$ which is strictly increasing except on a locally
finite number of segments, and  $\mathfrak{a}_0$ which is of the
Leray-Lions kind. We are interested in  existence, uniqueness and
stability of entropy  solutions. If $b=\mathrm{Id}$, we obtain a
general continuous dependence result on data $u_0,f$ and
nonlinearities $b,\psi,\phi,\tilde{\mathfrak{a}}$.
Similar result
is shown for the degenerate elliptic problem which corresponds to
the case of $b\equiv 0$ and general non-decreasing surjective
$\psi$.
Existence, uniqueness and continuous dependence on data $u_0,f$
are shown when $[b+\psi](\R)=\R$ and  $\phi\circ [b+\psi]^{-1}$ is
continuous.
\end{abstract}

\maketitle

\newpage
\tableofcontents

\section{Introduction}
\label{sec:intro}
\subsection{Problem and assumptions}
In this paper we consider problems under the general form\\[-4pt]
$$
  \leqno (P) \qquad
   \begin{cases}
    \ptl_t b(u)  +\div { \tilde{\mathfrak f}}(b(u),\psi(u),\phi(u))
   -\div {\mathfrak a}(u,\grad \phi(u))+\psi(u)=f,
     \\ \hspace*{70mm}
     \qquad \text{in $Q_T=(0,T)\times\Om$},
      \\  u|_{t=0}=u_{0} \quad \text{in $\Omega$}, \qquad
      u=0 \quad \text{on $
      (0,T)\times\ptl\Om$},
   \end{cases}
$$
where $u:(t,x)\in Q_T\longrightarrow\R$ is the
unknown function, $T>0$ is a fixed time, $\Omega\subset\R^N$ is a
bounded domain with Lipschitz boundary $\ptl\Om$.

We assume
$$
\leqno (H_1) \qquad \left|\begin{array}{l}
 \text{the functions $b,\psi,\phi: \R\mapsto
\R$ are continuous
nondecreasing,}\\
\text{normalized by the value zero at the point zero.}
\end{array}\right.
$$

We require the following technical assumption on $\phi$:\\[-5pt]
$$
\leqno (H_2)\qquad \left|\begin{array}{l} \text{there exists a
closed set $E\subset \R$ such that
$\phi$ is strictly increasing on $\R\setminus E$,}\\
\text{and the Lebesgue measure $\;\meas \phi(E)\;$ is zero};\\[-10pt]
\end{array}\right.
$$
and, moreover,\\[-5pt]
$$
\leqno (H_3) \qquad \liminf_{\eps\downarrow 0} \displaystyle \;
\meas(G^\eps)/\eps<+\infty, \quad \text{where $G^\eps:=\{z\in
\R\,|\, \dist(z,\phi(E))<\eps\}$.}
$$
Notice that since $\phi(\cdot)$ is
continuous and strictly monotone on $\R\setminus E$, the set
$G:=\phi(E)$ is also closed.
\begin{rem}\label{rem:setE'}
(i) Hypotheses $(H_2)$,$(H_3)$ are trivially satisfied if $\phi$
is a strictly increasing function. In case $\phi$ has a finite
number of segments on which it keeps constant values, $E$ is just
the union of all these ``flatness segments'', and $(H_2)$,$(H_3)$
are
satisfied.\\[3pt]
(ii) Property $(H_2)$ is still true if $\phi$ is locally
absolutely continuous. In general, the set of discontinuity points
of $\phi^{-1}$ is not closed, and its closure can be large (this
is the case, e.g., if $\phi$ is the ``Cantor stairs function'').
Thus $(H_2)$ is a restriction, although it is fulfilled in most of
the practical cases.  Property $(H_3)$ is a further restriction.
Indeed, consider the following example. It is easy to construct a
Lipschitz continuous non-decreasing function $\phi$ such that
$G=\phi(E)$ is equal to $\{0\}\cup\{ 1/\sqrt{i}\,|\, i\in \N\}$. A
straightforward calculation shows that for $\eps=1/n$, $G^\eps$
contains the whole interval $[0,1/(2n^{2/3})]$; in this case
$\meas(G^\eps)/\eps$ is of order $\eps^{-1/3}$ and gets unbounded
as $\eps\to 0$.
\end{rem}

 The initial function
$u_0:\Om\to\R$ and the source $f:Q\to \R$ are assumed to fulfill
$$
\leqno (H_4) \qquad \left|\begin{array}{l}
\text{$u_0\in L^\infty(\Om)$;\; 
$f$ is measurable such that}\\
\text{$f(t,\cdot)\in L^\infty(\Om)$ for a.e.~$t\in (0,T)$ and
$\dsp\int_0^T
\|f(t,\cdot)\|_{L^\infty(\Om)}\,dt<+\infty$.\footnotemark{}}
\end{array}\right.
$$
\footnotetext{In the sequel, we will abusively denote the latter
quantity by $\|f\|_{L^1(0,T;L^\infty(\Om))}$.} Furthermore, the
following condition (automatically satisfied in the case
$b(\R)=\R$) is needed :
$$
\leqno (H_5) \qquad \left| \begin{array}{l} \dsp \text{in the case
$b(+\infty)\neq+\infty$ (resp.,
$b(-\infty)\neq - \infty$) one has}\\[5pt]
\dsp  \psi(+\infty)=+\infty \text{ and } f^+\in L^\infty(Q_T)
 \\[5pt] \quad\Bigl(\text{ resp., } \psi(-\infty)=-\infty \text{ and } f^-\in L^\infty(Q_T)\;\Bigr).
\end{array}\right.
$$
Assumptions $(H_4)$,$(H_5)$ are imposed to limit our
study to bounded solutions of $(P)$.

 Note that in view of $(H_1)$ and
$(H_5)$, we are assuming at least that $(b+\psi)(\R)=\R$.
 An important part of the paper in devoted to the case $b(\R)=\R$. If, in
addition, $b$ is bijective, then performing a change of the
unknown one puts the problem into the doubly nonlinear framework
with $b=\text{Id}$.
%
%

 Our continuous dependence
result for problem $(P)$ (in which we perturb both the data and
the nonlinearities) concerns the case where the structure
condition
$$
\leqno (H_{str})
 \qquad b(r)=b(s)  \;
\Rightarrow \; \phi(r)=\phi(s)
$$
is satisfied. This
result implies the existence of solutions for $(P)$, by reduction
to non-degenerate problems. Assumption $(H_{str})$ is trivially
satisfied in the case $b=\text{Id}$.
 If $(H_{str})$ fails, the convergence of approximate
solutions to $(P)$ is known only for a particular monotone
approximation method developped by Ammar and Wittbold
\cite{AmmarWittbold}. This approach leads to an existence result
which bypasses $(H_{str})$; but interesting issues (such as
proving convergence of numerical methods for $(P)$ without
requiring the structure condition $(H_{str})$) remain open. See
B\'enilan and Wittbold \cite{BenilanWittbold} for a thoroughful
discussion of the role of the structure condition for a simple
model one-dimensional case $\partial_t
b(u)+(\mathfrak{f}(u))_x=u_{xx}$.

Furthermore,
$$
\leqno (H_6) \qquad \text{the function
$\tilde{\mathfrak{f}}:\R\times\R\times\R \to \R^N$ is assumed
merely continuous.}
$$
  Notice that under the
structure condition $(H_{str})$, the dependency of
$\tilde{\mathfrak{f}}$ on $\phi(u)$ can de dropped. Whenever it is
convenient (and in particular, in the case where $b$ is
bijective), we use the notation
${\mathfrak{f}}(\cdot):=\tilde{\mathfrak{f}}(b(\cdot),\psi(\cdot),\phi(\cdot))$.

The function  ${\mathfrak a}:\R\times\R^N\to \R^N$ is assumed to
satisfy the following conditions :\
$$
\leqno (H_7) \qquad \text { $\mathfrak{a}$ is
continuous\footnotemark{} on $\R\times\R^N$, and
$\mathfrak{a}(r,0)\equiv 0$};
$$
$$
\leqno (H_8) \qquad \left|\begin{array}{l} \text{$\mathfrak
a(r,\cdot)$ is monotone, i.e.
 }\\
 (\mathfrak a(r,\xi)-\mathfrak a(r,\eta))\cdot(\xi-\eta)\geq 0
 \quad \text{for all $\xi,\eta\in\R^N$ and all $r\in \R$};
 \end{array}\right.
$$
$$
\leqno (H_9) \qquad \left|\begin{array}{l} \text{$\mathfrak
a(r,\cdot)$ is coercive at zero, i.e., there exist
 $p\in(1,+\infty)$ and  $C\in C(\R;\R^+)$
 }\\
\text{such that } \; \mathfrak a(r,\xi)\cdot\xi\geq \frac
1{C(r)}|\xi|^p\;
 \quad \text{  for all $\xi\in\R^N$ and all $r\in \R$};
 \end{array}\right.
$$
$$
\leqno (H_{10}) \qquad \left|\begin{array}{l}
\text{the growth of $\mathfrak a(r,\xi)$ is not greater than $|\xi|^{p-1}$, i.e.,}\\
\text{there exists  $C\in C(\R;\R^+)$ such that } \\
 \; |\mathfrak a(r,\xi)|\leq C(r)(1+|\xi|^{p-1})\;
 \quad \text{  for all $\xi\in\R^N$ and all $r\in \R$}.
 \end{array}\right.
$$
  \footnotetext{ the assumption of continuity of
$\mathfrak{a}$ in the first variable can be relaxed: see
Remark~\ref{rem:NoSecondStructureCond}}
It follows from $(H_7)$-$(H_{10})$ that for all $r$, the operator
$w \mapsto -\div\, \mathfrak a(r,\grad w)$ is an operator acting
from $W^{1,p}_0(\Om)$ to $ W^{-1,p'}(\Om)$, where
$p'=\frac{p}{p-1}$. Since the work of Leray and Lions
\cite{LerayLions}, this assumption became classical. It can be
generalized to the framework of Orlicz spaces (see the works of
Ka\v{c}ur \cite{Kacur} and those of Benkirane and collaborators
(see e.g., \cite{BenkiraneElmahi},\cite{BenkiraneBennouna}), and
even to more general coercivity and growth assumptions. We refer
to Bendahmane and Karlsen \cite{BenKar:Renorm,BenKar:RenormII} for
the case of the anisotropic $p$-Laplacian. In the case of
dimension $N=1$, very general coercivity assumption
$\lim_{\xi\to\pm\infty} \mathfrak a(r,\xi)/\xi=+\infty$ (uniformly
for $r$ bounded) can be considered. In this framework, the
well-posedness for $(P)$ was already established by  Ouaro and
Tour\'e \cite{OuaroToure2002,OuaroToure2005,OuaroToure2007} and
Ouaro \cite{Ouaro2007} (see also B\'enilan and Tour\'e
\cite{BenilanToureIII}); notice that some essential arguments of
these works  are specific to the case $N=1$. Notwithstanding the
above generalizations, the classical Leray-Lions assumptions are
sufficient for us to illustrate the arguments of the existence
proof for $(P)$.

The relevant technical assumption in order to have uniqueness is

$$
\leqno (H_{11})
 \left|\begin{array}{l}
\text{there exists $C\in C(
\R^2;\R^+)$ such that }\;\\
 (\mathfrak a(r,\xi)-\mathfrak
 a(s,\eta))\cdot(\xi-\eta)+C(r,s)(1+|\xi|^p+|\eta|^p)|\phi(r)-\phi(s)|\geq 0\\
 \text{for all $\xi,\eta\in\R^N$ and all $r,s\in \R$ such that}\\
 \text{the segment
 which lies between $r$ and $s$}\\ \text{does not intersect the
 exceptional set $E$.}
  \end{array}\right.
$$
This assumption goes along the lines of Carillo and Wittbold
\cite{CarrilloWittbold}
 and combines the monotonicity condition
$(H_8)$ with a kind of Lipschitz continuity assumption on
$\mathfrak a(\cdot,\xi)\circ \phi^{-1}$ on the connected
components of $\R\setminus E$.

\begin{rem}\label{rem:NeglectedSet}
 Notice that  $E$ is a set of the values of $u$ that lead to
$\mathfrak a(u,\grad \phi(u))$ being zero. Indeed, since $\meas
\phi(E)=0$, we have $\grad \phi(u)=0$ a.e.~on the set where $u\in
E$; then we have $\mathfrak a(u,\grad \phi(u))=\mathfrak
a(u,0)=0$, regardless of the exact value of $u\in E$.
\end{rem}
\begin{rem}\label{rem:NoSecondStructureCond}
Notice that we do not assume the structure condition
$$
\phi(r)=\phi(s) \; \Rightarrow \;
\mathfrak{a}(r,\xi)=\mathfrak{a}(s,\xi) \qquad \forall
\xi\in\R^N.
$$
This means that $a(\cdot,\xi)$ can be discontinuous with respect
to $\phi(r)$; the set of discontinuities is included in $\phi(E)$
which,  by $(H_2)$, is a closed set of measure zero.
This technical assumption is needed to be able to ``cut off'' the
discontinuity set.

One can also consider $\mathfrak{a}(r,\xi)$ which is discontinuous
in $r$. E.g., take the case of $\mathfrak a(r,\xi)=\mathfrak
a_1(k(r),\xi)$ with $k(\cdot)$ piecewise continuous. Thanks to
Remark~\ref{rem:NeglectedSet}, it is reduced to our setting by a
change of unknown function $u$ into $v$ such that $u=\rho(v)$ with
$\rho$ non-strictly increasing, chosen so that $k(\rho(z))\equiv
\tilde k(z)$ with  $\tilde k(\cdot)$ continuous. Indeed, such
change of the unknown preserves assumption $(H_{str})$.
\end{rem}

Let us mention that the assumptions $(H_4)$,$(H_5)$
and $(H_9)$, $(H_{10})$ can also be generalized.
Within the framework of ``variational'' solutions, one usually
works with ``bounded energy initial data'', i.e., with $u_0$
measurable and such that $B(u_0)<+\infty$, where
\begin{equation}\label{eq:B-defi}
B(z):=\dsp \int_0^z \phi(s)\,db(s)
\end{equation}
is the function depending on $b$ and $\phi$ which we introduce
following Alt and Luckhaus \cite{AltLuckhaus}, and with relaxed
growth and coercivity assumptions allowing for additional terms
which depend on $B(r)+ |\phi(r)|^p+r\psi(r)$, these terms being
controlled by means of {\it a priori} estimates (see, e.g.,
\cite{AltLuckhaus,Kacur,AndrThesis}).

A more general framework is provided by the one of renormalized solutions (see
 \cite{BlanchardRedwane1,BlanchardRedwane2,CarrilloWittbold,AmmarWittbold,BlanchardPorretta,BenKar:Renorm} and the references therein).
  Nonetheless, the assumptions
$(H_4)$, $(H_5)$, $(H_9)$, and $(H_{10})$ are sufficiently weak
 to provide the starting point for the well-posedness theory for
renormalized solutions of $(P)$. Indeed, the uniqueness proof for
renormalized solutions of $(P)$ remains essentially the same as
the one of Carrillo and Wittbold in \cite{CarrilloWittbold}, and
the existence result is most easily obtained using bi-monotone
sequences of bounded entropy solutions, following Ammar and
Wittbold \cite{AmmarWittbold} and Ammar and Redwane
\cite{AmmarRedwane}.

\subsection{The notion of solution and known results}
Problem $(P)$ is of mixed elliptic-parabolic-hyperbolic type, and
thus combines the difficulties related to nonlinear conservation
laws with those related to nonlinear degenerate diffusion
equations. We refer to Kruzhkov \cite{Kruzkov} and to Leray and
Lions \cite{LerayLions}, Lions \cite{Lions}, Alt and Luckhaus
\cite{AltLuckhaus}, Otto \cite{Otto:L1_Contr} for the fundamental
works on these classes of equations, respectively.

One consequence is that the notion of weak solution (sometimes
called ``variational solution'') generally leads to
non-uniqueness, unless $\phi(\cdot)$ is strictly increasing.
 The notion of entropy solution we
use is adapted form the founding paper of Carrillo
\cite{Carrillo}, which extends the classical framework of entropy
solutions to scalar conservation laws to the case of problem $(P)$
with the linear diffusion $\mathfrak
a(u,\grad\phi(u))=\grad\phi(u)$. The uniqueness arguments of
\cite{Carrillo} were adapted by Carrillo and Wittbold
\cite{CarrilloWittbold} to the case of a nonlinear Leray-Lions
diffusion operator of the form $\mathfrak a(u,\grad u)$
corresponding to $\phi=\text{Id}$. The case of $b=\text{Id}$ and
of a diffusion operator of the form $\mathfrak a(\phi(u),\grad
\phi(u))$ is similar; one particular case is considered in
Andreianov, Bendahmane and Karlsen \cite{ABK}. Both frameworks are
sometimes referred to as ``doubly nonlinear''. Notice that a new
definition of an entropy solution, suitable also for the case of
doubly nonlinear anisotropic diffusion operators, was used in a
series of
 works by Bendahmane and Karlsen (see \cite{BenKar:Renorm,BenKar:RenormII} and
 references therein); the issue of existence in this general
 anisotropic framework is still open.
 The case of
triply nonlinear problems of the form $(P)$ has been first
addressed by  Ouaro and Tour\'e (see \cite{OuaroToure2007} and the
references therein) and Ouaro \cite{Ouaro2007}. Well-posedness
results are obtained in dimension $N=1$, under very general
coercivity conditions; see also the works of B\'enilan and Tour\'e
(\cite{BenilanToureIII} and the references therein).
 The multi-dimensional elliptic analogue of $(P)$ was recently addressed in the work Ammar and Redwane
\cite{AmmarRedwane},  in the framework of renormalized solutions;
their approach is quite similar to the ours, but the proofs of
\cite{AmmarRedwane} require a  special structure of the
nonlinearity $\phi$. We bypass the difficulties of
\cite{AmmarRedwane} using two observations in
Section~\ref{sec:MainLemmas} (see
Lemmas~\ref{SmallSetTruncationLemma},\ref{lem:WeakWeakOk}).

In most of the works cited hereabove, the homogeneous Dirichlet
boundary conditions were chosen. One should bear in mind that,
unless $\phi$ is strictly increasing, the boundary condition is
also understood in an entropy sense (see e.g. Bardos, LeRoux and
N\'ed\'elec \cite{Bardos:BC}, Otto \cite{Otto:BC}, Carrillo
\cite{Carrillo}). Focusing on the homogeneous boundary condition
is a simplification which seems to be not merely technical. A
partial extension of the Carrillo's techniques to inhomogeneous
boundary data can be found in Ammar, Carrillo and Wittbold
\cite{AmmarCarrilloWittbold}. Different techniques for the
inhomogeneous problem, based on the weak trace framework
introduced by Otto \cite{Otto:BC} and developed by Chen and Frid
\cite{ChFr1}, were used by Mascia, Porretta and Terracina
\cite{Mascia_etal:2000} and Michel and Vovelle
\cite{MichelVovelle}. A related, though somewhat more
straightforward approach was attempted by Andreianov and Igbida
\cite{AndrIgbida}. Notice that in all cases, a technique of
``going up to the boundary'' is preceded by obtaining the
fundamental weak formulations and entropy inequalities ``inside
the domain''. In the present paper, we also make the simplest
choice of the homogeneous Dirichlet data, and focus on deriving
the entropy inequalities.

\subsection{Main techniques and the outline of the paper}

  Our main concern is the existence for
$(P)$ (and more generally, the continuous dependence result with
respect to perturbations of the data and the nonlinearities) in
the case $b=\text{Id}$ (or, more generally, under the structure
condition $(H_{str})$). We extend the arguments of Andreianov,
Bendahmane and Karlsen \cite{ABK} developed for the case of
$\mathfrak{a}(u,\grad \phi(u))\equiv \mathfrak{a}(\grad \phi(u))$;
the main difficulty stems from the fact that we do not assume the
structure condition of Remark~\ref{rem:NoSecondStructureCond}, so
that $\mathfrak{a}(\phi^{-1}(\cdot),\xi)$ can be discontinuous.

In order to use the weak convergence while passing to the limit in
the nonlinear diffusion term in $(P)$, we use a version of the
classical Minty-Browder monotonicity argument. We use cut-off
functions to focalize on the intervals of the strict monotonicity
of $\phi$. We then show that the complementary of this set can be
neglected, thanks to a particular {\it a priori} estimate with the
cut-off function which localizes on a neigbourhood of the
``exceptional set'' $E$ of the values of $u$ introduced in $(H_2)$
(see Remark~\ref{rem:NeglectedSet} and
Lemma~\ref{SmallSetTruncationLemma}).

 In order to deal with the convection term in $(P)$,
we use the technique of nonlinear weak-$*$ convergence to a
measure-valued solution (as considered by Tartar, DiPerna,
Szepessy, Panov), or more exactly, to an entropy process solution
as developed by Gallou\"et and collaborators (see
\cite{GaHu,EyGaHe:book,EGHMichel} and references therein).
Considering entropy process solutions is a purely technical issue,
since the uniqueness proof also contains their identification to
entropy solutions.

The chain rule arguments of
Lemmas~\ref{lem:WeakWeakOk},\ref{lem:IBP-measures}  permit to
separate the two aforementioned weak convergence arguments, the
one for the diffusion term and the one for the convection term.

The uniqueness of an entropy solution is shown under the
assumption $(H_{11})$, with the help of
 $(H_3)$ and the estimate of Lemma~\ref{SmallSetTruncationLemma}; we follow Carrillo \cite{Carrillo}, Carrillo and Wittbold
\cite{CarrilloWittbold}, Eymard, Gallou\"et, Herbin and Michel
\cite{EGHMichel} and Andreianov, Bendahmane and Karlsen
\cite{ABK}.

While the full continuous dependence result strongly relies upon
the structure condition $(H_{str})$, the stability of entropy
solutions to $(P)$ with respect to the perturbation of the data
$(u_0,f)$ is shown under the weaker structure assumption
$$
\leqno (H'_{str}) \qquad \qquad (b+\psi)(r)=(b+\psi)(s) \;
\Rightarrow \; \phi(r)=\phi(s).
$$
 Let us stress the
fact that our argument using $(H'_{str})$ in not ``robust'', in
the sense that it cannot be directly adapted to the proof of
convergence of various kinds of approximate solutions.
 In
order to address the question of convergence of numerical
approximations of $(P)$, the stronger structure assumption
$(H_{str})$ still seems essential. For finite volume
approximations of the doubly nonlinear problem $(P)$ with
$b=\text{Id}$,  a convergence proof using $(H_{str})$ is given in
\cite{ABK}.

Let us give an outline of the paper. We start with definitions and
the formulation of the main results in
Section~\ref{sec:DefiResults}. In Section~\ref{sec:MainLemmas}, we
give the key ingredients of our techniques.
Section~\ref{sec:Uniqueness} concerns the adaptation of the
standard  uniqueness, $L^1$ contraction and comparison result for
entropy and entropy-process solutions of $(P)$.
 Section~\ref{sec:APrioriEstimates} contains the {\it a
priori } estimates for solutions. In
Section~\ref{sec:ContDependence}, we assume that $b$ is bijective
(or, more generally, conditions $(H_{str})$,$(H_5)$ are
satisfied); we deal with the convergence of solutions to problems
$(P_n)$ with perturbed coefficients.
 Section~\ref{sec:General-b} is
devoted to the proof of the well-posedness result for $(P)$.
 In Section~\ref{sec:Extensions} we give
existence, uniqueness and continuous dependence results for the
related elliptic equation $\psi(u) + \div { \tilde{\mathfrak
      f}}(\psi(u),\phi(u))-
      \div { \mathfrak a}(u,\grad \phi(u))=s\in L^\infty(\Om)$.

\section{Entropy solutions and well-posedness results}
\label{sec:DefiResults}

\subsection{Entropies and related notation}
As it was explained in the introduction, we need the notion of
weak solution for $(P)$ with additional ``entropy'' conditions. In
order to use entropy conditions in the interior of $Q_T$ and,
moreover, take into account the homogeneous Dirichlet boundary
condition, following Carrillo \cite{Carrillo} we will work with
the so-called ``semi-Kruzhkov'' entropy-entropy flux pairs
$(\eta^\pm_c,{\mathfrak q}^\pm_c)$ for each $c \in \R$:
$$
\eta^{+}_c (z)=(z-c)^{+}, \qquad \eta^{-}_c (z)=(z-c)^{-},
$$
$$
{\mathfrak q}^{+}_c(z)=\sgn{z-c} ({\mathfrak f}(z)-{\mathfrak
f}(c)), \qquad {\mathfrak q}^{-}_c(z)=\sgnm{z-c} ({\mathfrak
f}(z)-{\mathfrak f}(c)).
$$
By convention, we assign $(\eta^\pm_c)'(c)$ to be zero. Here
$(z-c)^\pm$ stand for the nonnegative quantities such that
$z-c=(z-c)^+ -(z-c)^-$, but we denote
\[
\sgn{z\!-\!c}=(\eta^+_c)'(z)=\left\{\ba{ll} 1,&\!\!\!\! z>c \\
0,&\!\!\!\! z\leq c \ea\right.\!\!, \quad
\sgnm{z\!-\!c}=(\eta^-_c)'(z)=\left\{\ba{ll} ~~~0,&\!\!\!\! z\geq
c
\\ -1,&\!\!\!\! z< c \ea\right.\!\!.
\]
At certain points, we will also need smooth regularizations of the
semi-Kruzhkov entropy-entropy flux pairs; it is sufficient to
consider regular ``boundary'' entropy pairs
$(\eta_{c,\eps}^\pm,\mathfrak{q}_{c,\eps}^\pm)$ (cf. Otto
\cite{Otto:BC} and the book \cite{Malek}), which are
$W^{2,\infty}$ pairs with the same support as
$(\eta_c^\pm,\mathfrak{q}_c^\pm)$, converging pointwise to
$(\eta_c^\pm,\mathfrak{q}_c^\pm)$ as $\eps\to 0$. In particular,
the functions
\[
\sign^+_\eps(z)= \frac{1}{\eps}\min\{z^+,\eps\}, \quad
\sign^-_\eps(z)=\frac{1}{\eps}\max\{-z^-,-\eps\}
\]
will be used to approximate
$\sign^\pm(\cdot)=(\eta_0^\pm)'(\cdot)$.
\begin{defi}\label{AThetaDef}
For a function $\varphi$ which is monotone on $\R$, for all
locally bounded piecewise continuous function $\theta$ on $\R$  we
can define (using, e.g., the Stiltjes integral)
\begin{equation*}
\varphi_\theta:z\in\R \mapsto \int_0^z \theta(s)\,d\varphi(s).
\end{equation*}
Moreover (see Lemma~\ref{FunctionalDepLemma} below), there exists
a continuous function $\widetilde \varphi_\theta$ such that $
\varphi_\theta(u)=\widetilde \varphi_\theta(\varphi(u)).$
\end{defi}
In the sequel, we denote by $b_c^\pm(\cdot)$ the function $\dsp
z\mapsto \int_0^z (\eta_c^\pm)'(s)\,db(s)$.

\subsection{Entropy and entropy process solutions}
For the sake of simplicity, we will in this paper work with
bounded entropy solutions, i.e., we require  that $u\in
L^\infty(Q_T)$ and put corresponding hypotheses on the data
$u_0,f$ and functions $b,\psi$. Note that the boundedness
assumption can be bypassed in the framework of renormalized
solutions (see in particular Ammar and Redwane
\cite{AmmarRedwane}); or in the framework of variational solutions
in the spirit of Alt and Luckhaus \cite{AltLuckhaus} (in this
case, one has to replace the functions $C(\cdot)$ in assumptions
$(H_9)$, $(H_{10})$, $(H_{11})$ with a constant $C$; further
changes are indicated in Remark~\ref{rem:bounded-b}).
\begin{defi}[entropy solution]
\label{def:Entropy-Solution} An entropy solution of $(P)$ is a
measurable function $u:\Do\to \R$ satisfying the following
conditions:
\begin{Definitions}
   \item (regularity) $u\in L^\infty(Q_T)$
   and
   $
   w=\phi(u)\in L^p(0,T;W^{1,p}_0(\Om)).
   $\\[-5pt]
   \label{def:entopy0}
   \item For all $\xi\in \D([0,T)\times \Om)$,
   \begin{equation}
      \label{eq:weakcond}
      \begin{split}
     \iint_{Q_T} \Biggl(b(u)\pt \xi
         & + {\mathfrak f}(u) \cdot \grad\xi
         - \mathfrak a(u,\Grad  w)\cdot
         \grad \xi -\psi(u)\xi\Biggl) \dx\dt
         \\ & +\int_\Omega b(u_0)\xi(0,\cdot)\dx
         + \iint_{Q_T}  f\xi \,dxdt=0.
      \end{split}
   \end{equation}
   \item For all $(c,\xi)\in \R^\pm\times\D([0,T)\times \overline{\Om})$, $\xi \geq 0$,
   and also for all $(c,\xi)\in \R\times\D([0,T)\times \Omega)$, $\xi \geq 0$,
   \begin{equation*}
      \begin{split}
         &\iint_{Q_T} \Biggl(b_c^\pm(u)\pt \xi
         + {\mathfrak q}_c^\pm(u) \cdot \grad\xi
         - (\eta_c^\pm)'(u) \mathfrak a(u,\Grad w)  \cdot
         \grad \xi -(\eta_c^\pm)'(u)\psi(u)\xi\Biggl) \dx\dt  \\ &  \qquad
         +\int_\Omega b_c^\pm(u_0)\xi(0,\cdot)\dx
         + \iint_{Q_T} (\eta_c^\pm)'(u)\, f\xi \,dxdt \ge 0.\\[-10pt]
      \end{split}
   \end{equation*}
   \label{def:entropy2}
\end{Definitions}
\end{defi}
\begin{rem}\label{rem:sub-super-solutions}
If in the above definition, $u$ satisfies {\rm \ref{def:entopy0}},
if {\rm \eqref{eq:weakcond}}  is replaced by the inequality
``$\geq$'' (resp., with the inequality ``$\leq$''), and if {\rm
\ref{def:entropy2}} holds with the entropies $\eta_c^+$ for $c\in
\R^+$ (resp., with the entropies $\eta_c^-$ for $c\in \R^-$), then
$u$ is called entropy subsolution (resp., entropy supersolution)
of $(P)$.
\end{rem}
\begin{rem}\label{rem:weak-AL-formulation}
Following Alt and Luckhaus \cite{AltLuckhaus}, we can rewrite the
weak formulation {\rm \eqref{eq:weakcond}} of $(P)$ as
follows:\\[3pt]
- $b(u)|_{t=0}=b(u_0)$ and the distributional
 derivative $\pt b(u)$ satisfies
$$
\pt b(u)\in L^{p'}(0,T;W^{-1,p'}(\Om))+L^1(Q_T)
$$
in the sense
$$
\int_0^T \langle \pt b(u)\,,\,\zeta\rangle
=-\iint_{Q_T} b(u)\,\pt\zeta -
\int_\Om b(u_0) \zeta (0,\cdot)
$$
for all $\zeta\in L^p(0,T,W^{1,p}_0(\Om))\cap L^\infty (Q_T)$ such
that $\pt \zeta\in L^\infty (Q_T)$ and $\zeta(T,\cdot)=0$;\\
- equation $(P)$ is satisfied in
$L^{p'}(0,T;W^{-1,p'}(\Om))+L^1(Q_T)$.\\[3pt]
We denote by $\langle \,\cdot\,,\,\cdot\,\rangle$ the
duality pairing between $L^{p'}(0,T;W^{-1,p'}(\Om))+L^1(Q_T)$ and
$L^p(0,T,W^{1,p}_0(\Om))\cap L^\infty (Q_T)$.
\end{rem}
 For technical reasons, it is convenient to introduce the notion
of entropy process solution adapted from Eymard, Gallou\"et and
Herbin \cite{EyGaHe:book},  Gallou\"et and Hubert \cite{GaHu} and
Eymard, Gallou\"et, Herbin and Michel \cite{EGHMichel}. This
definition is based upon the so-called ``nonlinear weak-$\star$
convergence'' property (see e.g., Ball \cite{Ball} and
Hungerb\"uhler~\cite{Hungerbuhler}):
\begin{equation}\label{nonlinweakstar}
\left|\begin{array}{l} \text{each sequence $(u_n)_n$ of measurable
functions}\\
\text{admits a subsequence such that for all $F\in C(\R,\R)$,}\\
\text{ $\; F(u_n(\cdot)) \to {\dsp\int_0^1}
F(\mu(\cdot,\alpha))\;d\alpha$ weakly in
$L^1(Q_T)$}\\
\text{whenever the set $(F(u_n))_n$ is weakly relatively compact
in $L^1(Q_T)$} ,
\end{array}\right.
\end{equation}
 where the function $\mu\in L^\infty(Q_T\times(0,1))$ is referred to
 as the
``process function''. Notice that in the above statement, one also
concludes that $F(\mu(\cdot,\alpha))$ independent of $\alpha$
whenever $F(u_n)$ converges to $F(u)$ in measure (see
Hungerb\"uhler \cite{Hungerbuhler}).
\begin{defi}[entropy process solution]
\label{def:Entropy-Solutionp} An entropy process solution of $(P)$
is a couple $(\mu,w)$ of measurable functions $\mu: Q_T\times
(0,1) \to \R$ and $w:Q_T\to\R$
 satisfying the following conditions:
\begin{DefinitionsP}
   \item (regularity and consistency) $\mu \in L^\infty(Q_T\times(0,1))$,
   $w\in L^p(0,T;W^{1,p}_0(\Om))$,\\[5pt]
   and
   $
   \phi(\mu(t,x,\alpha))\equiv w(t,x)
   $ for a.e.~$(t,x,\alpha)\in Q_T\times (0,1)$.\\[-5pt]
   \label{def:entropy1p}
\item For all $\xi\in \D([0,T)\times \Om)$,
   \begin{equation*}
     \begin{split}
         \int_0^1\iint_{Q_T} \Biggl( b(\mu)  \pt \xi
         & + {\mathfrak f}(\mu) \cdot \grad\xi - \mathfrak
         a(\mu,\grad w)\cdot \grad \xi -
          \psi(\mu)\xi
         \Biggr)\,dx\dt d\al
        \\ & +\int_\Omega u_0\xi(0,\cdot)\dx
        + \iint_{Q_T}  f\xi \,dxdt =0.
         \end{split}
   \end{equation*}
   \label{def:entropyWeakp}
   \item For all $(c,\xi)\in \R^\pm\times\D([0,T)\times \overline{\Om})$, $\xi \geq 0$,
   and also for all $(c,\xi)\in \R\times\D([0,T)\times \Omega)$, $\xi \geq 0$,
   \begin{equation*}
      \begin{split}
         &\int_0^1\!\!\! \iint_{Q_T} \!\!   \Biggl(\!  b^\pm_c(\mu)\pt \xi
         +  {\mathfrak q}_c^\pm(\mu) \cdot
         \grad\xi
         -  (\eta_c^\pm)'(\mu) \mathfrak a(\mu, \Grad w) \cdot
         \grad \xi - (\eta_c^\pm)'(\mu)\psi(\mu) \!\! \Biggr) \dx\dt d\al\\ &
         \qquad +\int_\Omega
         b^\pm_c(u_0)\xi(0,\cdot)\dx + \int_0^1\!\!\!\iint_{Q_T} (\eta^\pm_c)'(\mu)\,f\xi\,dxdtd\al \ge 0.
      \end{split}
   \end{equation*}
   \label{def:entropy2p}
\end{DefinitionsP}
\end{defi}
\begin{rem}\label{RemTwoFormsOfDiffusion} In {\rm \ref{def:entropy2p}}, setting $\dsp u:=\int_0^1\mu(\al)\,d\al$
one can rewrite the third term under the form
$$\dsp \int_0^1\!\!\!\iint_{Q_T}(\eta_c^\pm)'(\mu)
\mathfrak a(\mu, \Grad w)=
\iint_{Q_T}(\eta_c^\pm)'(u)\mathfrak{a}(u,\grad w).$$ Indeed, we
have $w\equiv \phi(\mu)$, and $\phi$ is invertible on $\R\setminus
E$, so that $\mu(t,x,\al)\equiv\phi^{-1}(w(t,x))=u(t,x)$ whenever
$w(t,x)\in \R\setminus \phi(E)$; furthermore, $\grad w=0$ a.e.~on
$[\,w\in\phi(E)\cup\{\phi(c)\}\,]$, and the exact value of
$(\eta_c^\pm)'(\mu)$ on $[\,w\in\phi(E)\cup\{\phi(c)\}\,]$ does
not matter because $\mathfrak a(r,0)\equiv 0$.
For the same reasons, $\mathfrak a(\mu, \Grad w)$ can be replaced
by $\mathfrak a(u, \Grad w)$ in {\rm \ref{def:entropyWeakp}}.
\end{rem}

\begin{rem}\label{rem:EPS-ES}
If $u$ is an entropy  solution of $(P)$, then the couple $(\mu,w)$
defined a.e.~on $Q_T\times(0,1)$ (resp., on $Q_T$) by
$\mu(t,x,\al)=u(t,x)$ (resp., by $w(t,x)=\phi(u(t,x))$), is an
entropy process solution of $(P)$. In turn, if
$(\mu,w)$ is an entropy process solution of $(P)$ such that
$\mu(t,x,\al)=u(t,x)$ a.e.~on $Q_T\times (0,1)$ for some
$u:Q_T\longrightarrow \R$, then the function $u$ is an entropy
solution of $(P)$.
\end{rem}

\subsection{Well-posedness of problem $(P)$ in the framework of entropy solutions.}
First note the uniqueness result, which requires no range
condition nor structure condition on the nonlinearities $b$ and
$\phi$.
\begin{theo}\label{th:EPSunique}
Assume that $(H_1)$-$(H_5)$ and
$(H_6)$-$(H_{11})$ hold.\\
(i) Assume that $(\mu,w)$ is an entropy process solution of $(P)$.
Then
$$
u(t,x)=\dsp \int_0^1 \mu(t,x,\al)\,d\al
$$
is an entropy solution of $(P)$.
Moreover, we have $b(\mu)(t,x,\al)\equiv
b(u)(t,x)$ and $\psi(\mu)(t,x,\al)\equiv \psi(u)(t,x)$ a.e.~on
$Q_T\times(0,1)$. If $(b+\phi+\psi)$ is strictly increasing, then
$\mu(t,x,\al)=u(t,x)$ a.e.~on $Q_T\times(0,1)$.
\\[3pt]
(ii) Assume that $u$ and $\hat u$ are two entropy solutions of
$(P)$ corresponding to the data $u_0,f$ and $\hat u_0, \hat f$,
respectively. Then for a.e.~$t\in (0,T)$,
\begin{align*}
    & \int_\Om (b(u)-b(\hat u))^+(t) + \int_0^t\!\!\int_\Om
    (\psi(u)-\psi(\hat u))^+
    \\ & \qquad \leq \int_\Om (b(u_0)-b(\hat u_0))^+
    +\int_0^t\!\!\int_\Om \sign^+(u-\hat u) (f-\hat f).
\end{align*}
(iii) In particular, if $u,\hat u$ are two entropy solutions of
$(P)$, then $b(u)\equiv b(\hat u)$ and $\psi(u)\equiv \psi(\hat
u)$.
\end{theo}
\begin{rem}\label{rem:comparaison}
 In Theorem~\ref{th:EPSunique}(ii), on can replace $u$, resp.
$\hat u$, with an entropy subsolution, resp. with an entropy
super-solution. The same proof applies.
\end{rem}
The following continuous dependence property is the central result
of this paper.
\begin{theo}\label{th:ContDep}
Let $(b_n,\psi_n,\phi_n,\mathfrak{a}_n,\tilde{\mathfrak{f}}_n;u^n_0,f_n)$, $n\in\N$, be
a sequence converging to $(b,\psi,\phi,\mathfrak{a},\tilde{\mathfrak{f}};u_0,f)$ in the
following sense:
\begin{equation}\label{eq:converg-data}
\begin{array}{ll}
\cdot & b_n,\psi_n,\phi_n \text{~converge pointwise
to~} b,\psi,\phi \text{~respectively};\\
\cdot & \tilde{\mathfrak{f}}_n,\mathfrak{a}_n \text{~converge to~}
\tilde{\mathfrak{f}},\mathfrak{a}, \text{~respectively,
uniformly on compacts};\\
 \cdot & b_n(u^n_0)\to b(u_0) \text{ in } L^1(\Om), \text{~and~} f_n\to f
 \text{ in } L^1(Q_T).
\end{array}
\end{equation}
Assume that $(b,\psi,\phi,\mathfrak{a},\tilde{\mathfrak{f}};u_0,f)$ and
$(b_n,\psi_n,\phi_n,\mathfrak{a}_n,\tilde{\mathfrak{f}}_n;u^n_0,f_n)$ (for each $n$) satisfy
the hypotheses $(H_1)$, $(H_4)$, $(H_5)$, and $(H_6)$-$(H_{11})$, and, moreover, that
the functions $C(\cdot)$ in $(H_9)$, $(H_{10})$,  and $(H_{11})$
as well as the $L^\infty(\Om)$ and $L^1(0,T,L^\infty(\Om))$
bounds in $(H_4)$ are independent of $n$.
We denote by $(P_n)$ the analogue of  problem $(P)$ corresponding
to the data and coefficients
$(b_n,\psi_n,\phi_n,\mathfrak{a}_n,\tilde{\mathfrak{f}}_n;u^n_0,f_n)$.

Assume that either $b(\R)=\R$, or the $L^\infty(Q_T)$ bounds on
$f_n^\pm$ in $(H_5)$ are independent of $n$. Assume that the
structure condition $(H_{str})$ holds, and $\phi$ satisfies the
technical hypotheses $(H_2)$,$(H_3)$.

Let $u_n$ be an entropy solution of problem $(P_n)$. Then the functions
$u_n$ converge to an entropy solution $u$ of $(P)$ in $L^\infty(Q_T)$ weakly-*, up to a
subsequence. Furthermore, the functions $\phi_n(u_n)$ converge to $\phi(u)$ in
$L^1(Q_T)$ up to a subsequence, and the whole sequences
$b_n(u_n),\psi_n(u_n)$ converge in $L^1(Q_T)$ to $b(u),\psi(u)$, respectively.
\end{theo}

\begin{rem}\label{rem:SimplerAssumptions-ContDep}
The structure condition $(H_{str})$ seems to be essential for results
like Theorem \ref{th:ContDep} to hold (see \cite{BenilanWittbold}).

The surjectivity assumption on $b$ or assumption $(H_5)$ were only
required to ensure the $L^\infty$ estimate on $u$. One
could work with unbounded solutions, in which case these
assumptions can be replaced by a growth condition on
$\tilde{\mathfrak{f}}$ (see Remark~\ref{rem:bounded-b} below).

Notice that assuming simultaneously $(H_{str})$, $b(\R)=\R$
and $\psi\equiv 0$, by a change of the unknown $u$
into $v=b(u)$ we can always reduce the triply
nonlinear problem $(P)$ to the doubly nonlinear problem with
$b\equiv \text{Id}$.
\end{rem}
Finally, we state the well-posedness result for $(P)$. Note that
when only the data $(u_0,f)$ are perturbed, the
continuous dependence result analogous to Theorem~\ref{th:ContDep}
holds under the structure assumption $(H'_{str})$ which is weaker
than $(H_{str})$.
\begin{theo}\label{th:EPS-ESexists}~\\
(i) Assume that $(H_1)$-$(H_5)$ and $(H_6)$-$(H_{11})$ hold. Then
there exists an entropy solution to $(P)$.
Moreover, it is unique, in the sense of Theorem~\ref{th:EPSunique}(iii).\\[3pt]
(ii) Assume in addition that the structure condition $(H'_{str})$
holds. Then the entropy solution of $(P)$ depends continuously on
the data $(u_0,f)$. More precisely, let $u_n$ be an entropy
solution of $(P)$ with data $(u_0^n,f_n)$. Assume $b(u^n_0)\to
b(u_0)$ in $L^1(\Om)$ and $f_n\to f$ in $L^1(Q_T)$, as $n\to
\infty$. Assume that  the bounds in $(H_4)$,$(H_5)$ are uniform in
the sense
\begin{itemize}
\item $\|u^n_0\|_{L^\infty(\Om)}\leq \Const{}$;\\[-5pt]

\item either $b(+\infty)=+\infty$ and  $\dsp\int_0^T \|f_n^+(t,\cdot)\|_{L^\infty(\Om)}\,dt\leq
\Const{}$, or $\psi(+\infty)=+\infty$ and $\|f^+_n\|_{L^\infty(Q_T)}\leq \Const{}$;\\[-5pt]

\item either $b(-\infty)=-\infty$ and  $\dsp\int_0^T
\|f_n^-(t,\cdot)\|_{L^\infty(\Om)}\,dt \leq \Const{}$, or
$\psi(-\infty)=-\infty$ and $\|f^-_n\|_{L^\infty(Q_T)}\leq
\Const{}$.
\end{itemize}
Then $b(u_n),\psi(u_n)$ and $\phi(u_n)$ converge, respectively, to
$b(u),\psi(u)$ and $\phi(u)$ in $L^1(Q_T)$ as $n\to \infty$, where
$u$ is an entropy solution of $(P)$ with data $(u_0,f)$.

Moreover, if we reinforce hypothesis $(H_8)$ by requiring
the uniform monotonicity of $\mathfrak{a}(r,\cdot)$ in the sense
$$
\leqno (H'_8) \qquad \left|
\begin{array}{l}
\text{there exists $C\in C( \R^2;\R^+)$ such that }\\
(\mathfrak a(r,\xi)-\mathfrak a(r,\eta))\cdot(\xi-\eta)\geq
1/C(r,\frac 1{|\xi-\eta|}),
\end{array} \right.
$$
then also  $\grad \phi(u_n)$ converge to $\grad
\phi(u)$ in $(L^p(Q_T))^N$ and $\mathfrak{a}(u_n,\grad\phi(u_n))$ converge to
 $\mathfrak{a}(u,\grad\phi(u))$ in $(L^{p'}(Q_T))^N$, as $n\to \infty$.
\end{theo}

\section{Notation and preliminary lemmas}\label{sec:MainLemmas}

Let us give some notation which will
be used throughout the remaining sections.

We will use the notation like $\Bigl[u\in F\Bigr]$ for the sets
like $\{(t,x)\in Q_T\,|\,u(t,x)\in F\}$. For a  measurable set
$H$, we denote by $\char_{H}$ the characteristic function of $H$.
For $H\subset\R$, we set
$$
T_H(z):=\dsp\int_0^z \char_{H}(s)\,ds;
$$
clearly, $T_H(\cdot)$ is a Lipschitz function with $T_H(0)=0$.

We denote by $G$ the image $\phi(E)$ by $\phi(\cdot)$ of the
``exceptional set'' $E$; recall that $E$ is closed, $G$ is closed
and $\meas(G)=0$. We denote by $I$ a generic open interval in
$\R\setminus E$, and by $J$ it image $\phi(I)$ which is a generic
open interval in $\R\setminus G$. For all $\eps>0$, we choose an
open set $G_\eps\supset G$ such that
$\meas(G_\eps)<\Const{}\times\eps$. We denote by $E_\eps$ the open
set $\phi^{-1}(G_\eps)$ which contains $E$.
 When $(H_3)$ holds, we
can simply take $G_\eps=G^\eps:=\{z\in \R\,|\, \dist(z,G)<\eps\}$.

Now let us prove the representation property used in Definition~\ref{AThetaDef}.
%
\begin{lem}\label{FunctionalDepLemma}
Let $\varphi_\theta(\cdot)$ be the function defined by
$$
\dsp \varphi_\theta:z\in\R \mapsto \int_0^z \theta(s)\,d\varphi(s),
$$
for a  continuous non-decreasing function $\varphi:\R\to\R$ and a bounded piecewise
continuous function $\theta:\R\to\R^n$. Then there exists a
Lipschitz continuous function $\widetilde \varphi_\theta:\varphi(\R)\to \R^n$
such that for all $z\in \R$,
$$
\varphi_\theta(z)=\widetilde \varphi_\theta(\varphi(z)).
$$
\end{lem}
\begin{proof}
 If $\varphi(z)=\varphi(\hat z)$, then the measure
$d\varphi(s)$ vanishes between $z$ and $\hat z$; thus
$\varphi_\theta(z)-\varphi_\theta(\hat z)=\dsp\int_{\hat z}^{z}
 \theta(s)\,d\varphi(s)$ is zero. Therefore $\widetilde \varphi_\theta$ is well defined.
 For all $r,\hat r\in \varphi(\R)$,
 $\widetilde \varphi_\theta(r)-\widetilde \varphi_\theta(\hat r)=\varphi_\theta(z)-\varphi_\theta(\hat z)=\dsp\int_{\hat
 z}^z \theta(s)\,d\varphi(s)$, where $z\in \varphi^{-1}(r)$, $\hat z \in \varphi^{-1}(\hat
 r)$. Thus $$|\widetilde \varphi_\theta(r)-\widetilde \varphi_\theta(\hat
 r)|\leq
  \|\theta\|_{L^\infty}\;|\varphi(z)-\varphi(\hat z)|= \|\theta\|_{L^\infty}\;|r-\hat
  r|.\\[-15pt]$$
\end{proof}

Now, let us give a localized estimate of the gradient of $w=\phi(u)$.
\begin{lem}\label{SmallSetTruncationLemma}
Let $u$ be a bounded weak solution of $(P)$. Then there exists a
constant $C$ depending on $C(\cdot)$ in $(H_9)$, on
$\|u\|_{L^\infty(Q_T)}$ and on $\|b(u_0)\|_{L^1(\Om)}$,
$\|f\|_{L^1(Q)}$ such that for all Borel measurable set $F\subset
\R$,
\begin{equation}\label{eq:LocalEstimate}
    \begin{split}
        I_F(u) & :=\iint_{\Bigl[u\in F\Bigr]} |\grad
        \phi(u)|^{p}=\iint_{\Bigl[w\in \phi(F)\Bigr]} |\grad w|^{p}
        \\ & \leq C\, \text{\rm Var}_{F} \phi(\cdot)=C \,\meas (\phi(F)).
    \end{split}
\end{equation}
\end{lem}

\begin{proof}
Without loss of restriction, one can assume that  $F$ is bounded;
indeed, otherwise we can replace $F$ with $F^M:=F\cap [-M,M]$ and
then pass to the limit as $M\to+\infty$ in
inequality \eqref{eq:LocalEstimate} written for $F^M$.

Set $H=\phi(F)$, and note that $T_H(w)=T_H(\phi(u))\in L^p(0,T;W^{1,p}_0(\Om))$ can be
approximated by admissible test functions in \eqref{eq:weakcond} of
Definition \ref{def:Entropy-Solution}; one has
$$
\grad T_H(w)=\grad w \char_{H}(w)
=\grad \phi(u) \char_{F}(u),
$$
and
\begin{equation}\label{EstTalpha}
\|T_H(\cdot)\|_\infty \leq \int_{F} d\phi(s)
=\text{\rm Var}_{F}
\phi(\cdot)=\meas (H).
\end{equation}

Using this test function, with
Remark~\ref{rem:weak-AL-formulation} and the standard chain rule
argument known as the Mignot-Bamberger and Alt-Luckhaus formula
(see, e.g., Alt and Luckhaus \cite{AltLuckhaus}, Otto
\cite{Otto:L1_Contr}, Carrillo and Wittbold
\cite{CarrilloWittbold}) we get
\begin{equation}\label{EstWithTalpha}
\begin{array}{l}
\dsp \int_\Om B_F(u)(T,\cdot)+ \iint_{Q_T} \psi(u)T_H(\phi(u)) +
\iint_{Q_T} \mathfrak a(u,\grad w)\cdot \grad
T_H(w) \\[5pt]
\dsp \qquad\qquad\qquad =\int_\Om B_F(u_0)+ \iint_{Q_T} f T_H(w) -
\iint_{Q_T} {\mathfrak f}(u)\cdot \grad T_H(\phi(u)) ,
\end{array}
\end{equation}
where $B_F(z)=\dsp \int_0^z T_H(\phi(s))\,db(s)$. The last term is
zero thanks to the boundary condition $w|_{\Sigma}=0$. Indeed,
because $F$ is assumed bounded, $\mathfrak{f}(\cdot)$ is bounded
on the support of $T_H'(\phi(\cdot))$; thus by
Lemma~\ref{FunctionalDepLemma} there exists a Lipschitz continuous
vector-valued function ${\mathfrak g}(\cdot)$ such that
$$
\dsp \int_0^z {\mathfrak f}(s)\,dT_H(\phi(s)) =\mathfrak g(\phi(z)).
$$
Hence $\mathfrak g\circ w\in L^p(0,T;W^{1,p}_0(\Om;\R^N))$, so
that one can apply the Green-Gauss formula to get
$$
\iint_Q \div \Bigl( \int_0^w {\mathfrak f}(s)\,dT_H(\phi(s)) \Bigr)
=\int_0^T\!\int_{\ptl\Om}  {\mathfrak g} (w)\cdot \nu=0,
$$
where $\nu$ is the exterior unit normal vector to $\ptl\Om$. By
definition of $T_H(\cdot)$ and because $\phi(\cdot)$ is
non-decreasing, dropping positive terms in the left-hand side of
\eqref{EstWithTalpha}, by $(H_9)$ we infer
\begin{align*}
\frac{1}{C(\|u\|_{L^\infty(Q_T)})}
 \iint_{[u\in F]} \Bigl|\grad \phi(u)\Bigr|^{p}
& \leq  \iint_{[w\in \phi(F)]}  \mathfrak a(u,\grad w)\cdot \grad w
\\ & \leq \Bigl(\|b(u_0)\|_{L^1(\Om)}
+\|f\|_{L^1(Q)}\Bigr)\, \|T_H(\cdot)\|_\infty.
\end{align*}
Hence the claim follows by \eqref{EstTalpha}.
\end{proof}

In the above proof, we have used two chain rule lemmas. Now we
notice that both apply
for $u(\cdot)$ replaced with a ``process function''
$\mu(\cdot,\alpha)$, as in {\rm \ref{def:entropy1p}}, provided that for a.e.~$\al\in(0,1)$
one can substitute
$$
u(t,x):=\dsp\int_0^1\mu(t,x,\al)\,d\al
 $$
by $\mu(t,x,\al)$ in the expression of the test function.
\begin{lem}\label{lem:WeakWeakOk}
Let $(\mu,w)$ satisfy {\rm \ref{def:entropy1p}}, and
$S:\R\to \R$ be a Lipschitz continuous function such that
$S(0)=0$. Let $\zeta\in L^\infty(0,T)$.
Then
$$
\iint_{Q_T}\!\int_0^1 {\mathfrak f}(\mu(t,x,\alpha)))
\cdot\grad S(w(t,x))\;\zeta(t)\; dtdxd\alpha=0.
$$
\end{lem}
\begin{proof}
By Lemma~\ref{FunctionalDepLemma}, there exists a Lipschitz
vector-valued function $ \mathfrak{g}$ such that $\dsp\int_0^{z}
{\mathfrak f}(s)\,dS(\phi(s))={\mathfrak{g}}(\phi(z))$ for
$|z|\leq \|\mu\|_{L^\infty(\Do\!\times\!(0,1))}$. We have
\begin{align*}
    &\iint_{Q_T}\!\int_0^1 {\mathfrak f}(\mu(\al))\,d\al\cdot\grad S(\phi(u))\,\zeta
    \\ & =\iint_{Q_T}\!\int_0^1 {\mathfrak f}(\mu(\al))\cdot \grad
    S(\phi(\mu(\al)))\,d\al\,\zeta \\
    & = \iint_{Q_T}\!\int_0^1 \div\, \Bigl(\int_0^{\mu(\al)}
    {\mathfrak f}(s)\,dS(\phi(s)) \Bigr)\,d\al\,\zeta
    \\ & =\int_0^T\!\!\!\int_\Om \div\, {\mathfrak{g}}(w)\,\zeta
    = \int_0^T\!\!\!\int_{\ptl\Om}{\mathfrak{g}}(w) \cdot \nu\;\zeta=0
\end{align*}
because for a.e.~$\alpha\in (0,1)$, $\phi(\mu(\alpha))\equiv
\phi(u)=w\in L^p(0,T;W^{1,p}_0(\Om))$.
\end{proof}
\begin{lem}\label{lem:IBP-measures} Let $\Om$ be a bounded domain
of $\R^n$, $T>0$, $Q_T:=(0,T)\times\Om$, and $1<p<+\infty$. Let
$g\in C(\R;\R)$. Let $b\in C(\R; \R)$ be non-decreasing.
Set
$$
B_g(z):=\int_0^z g(s)\,db(s).
$$

Let $\mu\in L^\infty (Q_T\times(0,1))$; set $u=\dsp\int_0^1
\mu(\al)\,d\al$. Assume that
$$
g(u)\in L^p(0,T;W^{1,p}_0(\Om))\cap L^\infty(Q_T)
$$
and, moreover,
$$
g(\mu(\al))\equiv g(u).
$$
Assume that
$$
\pt \biggl(\int_0^1 b(\mu(\al))\,d\al\biggr)\in
L^{p'}(0,T;W^{-1,p'}(\Om))+L^1(Q_T)
$$
and
$$
\dsp\int_0^1 b(\mu(\al))\,d\al|_{t=0}=b(u_0)
$$
in the following sense:
$$
\begin{array}{l}
\dsp \text{$\forall \xi\in L^p(0,T;W^{1,p}_0(\Om))$ such that $\pt
\xi\in L^\infty(Q_T)$ and $\xi(T,\cdot)=0$},\\[8pt]
\dsp \int_0^T \langle \dsp\pt \biggl(\int_0^1\!
b(\mu(\al))\,d\al\biggr)\,,\,\xi\rangle =-\iint_{Q_T}\! \int_0^1\!
b(\mu(\al))\,d\al\,\pt \xi-\int_\Om b(u_0)\xi(0,\cdot).
\end{array}
$$
Then for all $\zeta \in \mathcal D([0,T))$,
$$
\int_0^T\!\! \langle \pt \biggl(\int_0^1
\!\!b(\mu(\al))\,d\al\biggr)\,,\, g(u)\zeta\rangle = -\iint_{Q_T}
\!\int_0^1 \!\!B_g(\mu(\al))\,d\al\, \zeta_t-\int_\Om \!B_g(u_0)
\zeta(0).
$$
\end{lem}
\begin{proof}[Proof (sketched).]
Note that the claim of Lemma~\ref{lem:IBP-measures} cannot be
deduced directly from the usual Mignot-Bamberger and
Alt-Luckhaus chain rule lemma; the reason is that we cannot expect
$\pt b(\mu(\al))$ to belong to $L^{p'}(0,T;W^{-1,p'}(\Om))+L^1(Q_T)$ for a.e.~~$\alpha$.
But it suffices to reproduce the proof (see, e.g., \cite{CarrilloWittbold})
which is by discretization of $\pt \biggl(\int_0^1 b(\mu(\al))\,d\al\biggr)$.
Indeed, we have for a.e.~$t,t-h\in(0,T)$,
\begin{align*}
& \frac 1h \left( \int_0^1 b(\mu(t,\al))\,d\al
- \int_0^1\! b(\mu(t\!-\!h,\al)) \,d\al\right)\;g(u(t))
\\ & = \int_0^1\! \frac 1h \Bigl( b(\mu(t,\al))-b(\mu(t\!-\!h,\al))\Bigr)
\,g(\mu(t,\al))\;d\al,
\end{align*}
and now we can reason separately for each $\alpha$. Thus
the arguments of \cite{CarrilloWittbold} apply.
\end{proof}

\section{Proof of $L^1$ contraction and comparison principles}
\label{sec:Uniqueness}

Now we turn to the proof of Theorem~\ref{th:EPSunique} and
Remark~\ref{rem:sub-super-solutions}.   Most of the statements are
standard. We only notice that while proving
Theorem~\ref{th:EPSunique}(i), one obtains that $b(\mu)$ and
$\psi(\mu)$ are independent of $\al$; since $\phi(\mu)=w$ is
independent of $\al$ by definition, one concludes that $\mathfrak
f(\mu)\equiv \tilde{\mathfrak{f}}(b(\mu),\psi(\mu),\phi(\mu))$ is
also independent of $\alpha$, thus the entropy process solution
$\mu$ gives rise to the entropy solution
$u=\int_0^1\mu(\al)\,d\al$. This is the only point where the
special structure of the dependency of $\mathfrak f$ on $u$ is
used.

The proof of Theorem~\ref{th:EPSunique} is essentially the same as
in Carrillo and Wittbold \cite{CarrilloWittbold}; it is based on
the techniques of Carrillo \cite{Carrillo} and on hypothesis
$(H_{11})$ (notice that in the case $\phi=\text{Id}$, one has
$E=\O$; therefore $(H_{11})$ reduces to the Carrillo-Wittbold
hypothesis in this case). For Theorem~\ref{th:EPSunique}(i), we
also need the adaptation of the Carrillo arguments to the
framework of entropy process solutions. This has been done by
Eymard, Gallou\"et, Herbin and Michel \cite{EGHMichel}, Michel and
Vovelle \cite{MichelVovelle} and Andreianov, Bendahmane and
Karlsen \cite{ABK}.
 Therefore we only point
out why hypothesis $(H_{11})$ is sufficient for the uniqueness of
an entropy solution in the case of problem $(P)$ with $\phi$ that
can be not strictly increasing.

The role of hypothesis  $(H_{11})$ is to ensure that
\begin{equation}\label{eq:mechante-integrale}
 \limsup_{\eps\to 0}  \iint_{Q_T}\!\iint_{Q_T}\frac {1}{\eps}\,
\Bigl(\,\mathfrak{a}(u,\grad w)-\mathfrak{a}(\hat u,\grad \hat
w)\,\Bigr)\cdot\Bigl(\grad w - \grad \hat w \Bigr)
\;\char_{[0<w-\hat w<\eps]}\geq 0,
\end{equation}
where
$$
u=\int_0^1 \mu(\al)\,d\al,\;\; w=\phi(u), \quad \hat u=\int_0^1
\hat \mu(\al)\,d\al,\;\; \hat w=\phi(\hat u),
$$
and  $(\mu(t,x,\al),w(t,x))$ and $(\hat \mu(s,y,\al),\hat w(t,x))$
are two entropy process solutions of $(P)$. Here, following
Kruzhkov \cite{Kruzkov}, we have taken two independent sets of the
variables $(t,x)$ and $(s,y)$.

We split the integration domain $Q_T\times Q_T$
 into several pieces.

First, notice that a.e.~on $[w\in G]\times Q_T$, we have $\grad
w=0$; thus the integrand in \eqref{eq:mechante-integrale} is
reduced to $\mathfrak{a}(\hat u,\grad \hat w) \grad \hat w
\;\char_{[0<w-\hat w<\eps]}$, which is non-negative. The same
argument applies on $Q_T\times [\hat w\in G]$.

Thus it remains to investigate the integrand in
\eqref{eq:mechante-integrale} on the set $[w\notin G]\times[\hat
w\notin G]$. Let us introduce $G^\eps:=\{z\in \R\,|\,
\dist(z,G)<\eps\}$. For a.e.~$(t,x,s,y)\in [w\notin G]\times[\hat
w\notin G]$ we have
:\\[3pt]
(a) either $w(t,x)$ and $w(s,y)$ belong to the same connected
component of $\R\setminus G$;\\
(b) or (a) fails, but $w(t,x)\in G^\eps\setminus G$ and $\hat w(s,y)\in G^\eps\setminus G$;\\
(c) or both (a) and (b) fail, but then $|w(t,x)-\hat w(s,y)|\geq
2\eps$.\\[3pt]
We then split $[w\notin G]\times[\hat w\notin G]$ into disjoint
union of sets $S_a\cup S_b \cup S_c$, according to which of the
above cases (a),(b),(c) takes place at $(t,x,s,y)\in [w\notin
G]\times[\hat w\notin G]$. On $S_a$, we use assumption $(H_{11})$
and infer that the integrand in \eqref{eq:mechante-integrale} is
minorated by
$$
-\max
\{C(r,s)\;\Bigl|\,\;|r|,|s|\leq\|u\|_\infty\}\;(1+|\grad
w|^p+|\grad \hat w|^p)\,\char_{[0<w-\hat w<\eps]}.
$$
Because the $2(N\!+\!1)$-dimensional Lebesgue measure of the set
$[0<w-\hat w<\eps]$ goes to zero as $\eps\to 0$, the limit of the
corresponding part of the integral in
\eqref{eq:mechante-integrale} is minorated by zero.

On $S_b$, we bound the integrand in \eqref{eq:mechante-integrale}
from below by $-\frac p\eps\,(|\grad w|^p+|\grad \hat w|^p)$.
Using Lemma~\ref{SmallSetTruncationLemma} we have e.g.
$$
\frac 1\eps\iint_{[w\in G^\eps\setminus G]}\!\!\iint_{[\hat w\in
G^\eps\setminus G]} |\grad w|^p \leq  C \frac{\meas(G^\eps)}{\eps}
\;\iint \char_{[\hat w\in G^\eps\setminus G]}.
$$
By the continuity of the Lebesgue measure and because
$\cup_{\eps>0} G^\eps\setminus G=\O$,  the measure of the set
$[\hat w \in G^\eps\setminus G]$ tends to zero as $\eps\to 0$.
Therefore, using  assumption $(H_3)$, we deduce that the
corresponding part of the limit in \eqref{eq:mechante-integrale}
is non-negative.

Finally, on $S_c$ the integrand in \eqref{eq:mechante-integrale}
is zero.

This ends the proof of \eqref{eq:mechante-integrale}.

\section{A priori estimates}
\label{sec:APrioriEstimates}

The following estimates are rather standard.
\begin{lem}\label{lem:apriori-estimates}
Let $(b_n,\psi_n,\phi_n,\mathfrak{a}_n,\tilde{\mathfrak{f}}_n;u^n_0,f_n)$,
$n\in\N$ be a sequence of data satisfying the assumptions of
Theorem~\ref{th:ContDep}. Assume that the limiting data
$(b,\psi,\phi,\mathfrak{a},\tilde{\mathfrak{f}};u_0,f)$ are such
that $(H_5)$ and $(H_{str})$ hold.

Let $u_n$ be an entropy solution of problem $(P_n)$. Then there
exists a constant $M$ and a modulus of continuity
$\omega:\R^+\to\R^+$,  such that for
all $n\in\N$,\\[3pt]

(i) $\|u_n\|_{L^\infty(Q_T)}\leq M$;
\\[3pt]

 (ii) the following quantities are all upper bounded by $M$:
\begin{align*}
&\|\phi_n(u_n)\|_{L^p(0,T;W^{1,p}_0(\Om))}, \,
\|\phi_n(u_n)\|_{L^p(Q_T)}, \, \|\mathfrak{a}_n(u_n,\grad\phi_n(u_n))\|_{L^{p'}(Q_T)}, \,
\\ &  \|\psi_n(u_n)\phi_n(u_n)\|_{L^1(Q_T)}, \,
\|B_n(u_n)\|_{L^\infty(0,T;L^1(\Om))}
\end{align*}
where $B_n$ is  defined in \eqref{eq:B-defi}
with $b,\phi$ replaced by $b_n,\phi_n$;
\\[3pt]

 (iii) $\text{ for all $\Del>0$,}\quad
\dsp \iint_{Q_{T-\Del}}
|\phi_n(u_n(t\!+\!\Del,x))\!-\!\phi_n(u_n(t,x))| \;\leq\,
\omega(\Del).$
\end{lem}
\begin{proof}
(i) First assume $b(\R)=\R$. Consider the function
$$M(t):=\sup_{n\in\N} \Bigl( \|b_n(u^n_0)\|_{L^\infty(\Om)} +\int_0^t
\|f_n(\tau,\cdot)\|_{L^\infty(\Om)}\,d\tau\Bigr)<+\infty.$$
 Then for any measurable choice of
  $\overline{u}(t,x) \in b_n^{-1}(M(t))$, $\overline{u}$ is an entropy supersolution of
$(P_n)$. Similarly, $\underline{u}(t,x) \in b_n^{-1}(-M(t))$ is an
entropy subsolution of $(P_n)$. The comparison principle of
Remark~\ref{rem:comparaison} ensures that a.e.~on $Q$,
$$-M(T)\leq -M(t)\leq b_n(u_{n})(t,x) \leq M(t)\leq M(T).$$
Now the assumption $b(\R)=\R$ and the pointwise convergence of
$b_n$
to $b$ ensure the uniform $L^\infty(Q_T)$ bound on $u_n$.

If $b(+\infty)<+\infty$, then $(H_5)$ ensures that any
constant $\underline{u}\in \psi_n^{-1}(\|f^+\|_{L^\infty(Q_T)})$
is an entropy subsolution of $(P_n)$. As hereabove, the
comparison principle and the pointwise convergence of $\psi_n$ to
$\psi$ satisfying $\psi(+\infty)=+\infty$ yield a uniform
majoration of $u_n$. The
case $b(-\infty)>-\infty$ is analogous.\\[3pt]

(ii) We use the test function $\phi_{n}(u_{n})$ in the weak
formulation of $(P_n)$. The duality product between
$$
\phi_n(u_{n})\in L^{p}(0,T;W^{1,p}_0(\Om))\cap L^\infty(Q_T)
$$
and
$$
\pt b_n(u_{n})\in L^{p'}(0,T;W^{-1,p'}(\Om))+L^1(Q_T)
$$
is treated via the standard chain rule argument
(\cite{AltLuckhaus,Otto:L1_Contr,CarrilloWittbold}). Using in
addition the chain rule  of Lemma~\ref{lem:WeakWeakOk}, the
$L^\infty$ bound on $u_n$ shown in (i), the uniform coercivity
assumption $(H_9)$, and the obvious inequality $B_n(z)\leq
b_n(z)\phi_n(z)$, we obtain the inequality
\begin{align*}
& \int_\Om B_n(u_n(t,\cdot))
+ \int_0^t\!\!\int_\Om \biggl(\psi_n(u_n)\phi_n(u_n)+ c|\grad \phi_n(u_n)|^p \biggr)
\\ & \qquad \leq
\int_\Om b_n(u^n_0)\,\phi_n(u_0^n) + \int_0^t\!\!\int_\Om f_n\, \phi_n(u_n)
\end{align*}
with some $c>0$ independent of $n$. Notice that the functions
$b_n,\phi_n$ are locally uniformly bounded because they are
monotone and converge pointwise to $b,\phi$, respectively.
Therefore the right-hand side of the above inequality is bounded
uniformly in $n$, thanks to (i) and to the uniform bounds on the
data $u_0^n$ and $ f_n$ in $L^\infty(\Om)$ and in $L^1(Q_T)$,
respectively. The uniform estimate of the left-hand side follows.
We then estimate $\|\phi_n(u_n)\|_{L^p(Q_T)}$  by the Poincar\'e
inequality; the $(L^{p'}(Q_T))^N$ bound on
$\mathfrak{a}_n(u_n,\grad\phi_n(u_n))$ follows from the growth assumption $(H_{10})$.\\[3pt]

(iii) Let $\Del>0$. The weak formulation of $(P_n)$ yields, for
a.e.~$t,t+\Del\in (0,T)$,
\begin{align*}
&\int_\Om (b_n(u_{n})(t+\Del)-b_n(u_{n})(t))\,\xi
\\ & \quad =
\int_t^{t\!+\!\Del}\!\!\!\! \int_\Om \Bigl[\;\Bigl(\,-{\mathfrak f}_{n}(u_{n})
+{\mathfrak a}(u_n,\grad \phi_{n}(u_{n}))\,\Bigl)
\cdot \grad \xi\; -\; \psi_n(u_n)\,\xi
\;+\; f\,\xi\;\Bigr]
\end{align*}
for all $\xi\in W^{1,p}_0(\Om)\cap L^\infty(\Om)$ (here and in the
sequel, we drop the dependence on $x$ in the notation). We take
$\xi=\phi_{n}(u_{n}(t+\Del))-\phi_{n}(u_{n}(t))$ and integrate in
$t$, then use the Fubini theorem which makes appear the factor
$\Del$ in the right-hand side; with  the estimates of (i),(ii) and
the uniform bounds on the data, we deduce that
\begin{equation}\label{Transl-base}
 \iint_{Q_{T-\Del}}
|b_n(u_{n})(t+\Del)-b_n(u_{n})(t)|\;|\phi_{n}(u_{n})(t+\Del)-\phi_{n}(u_{n})(t)|
\leq d \,|\Del|.
\end{equation}
Here $d$ is a generic constant independent of $n$. In the sequel,
denote by $(r_n)_n$ a generic sequence vanishing as $n\to\infty$.
Notice that by the Dini theorem, $b_n,\phi_n$ converge to
$b,\phi$, respectively, uniformly on compact subsets of $\R$. By
(i), it follows that we can replace $b_n,\phi_n$ in
\eqref{Transl-base} by $b,\phi$, provided that a term $r_n$ is
added to the right-hand side of \eqref{Transl-base}.

Now notice that assumption $(H_{str})$ implies that
$\tilde\phi:=\phi\circ b^{-1}$ is a continuous function. Let $M$
be the $L^\infty$ bound for $u_n$ in (i). Let $\pi$ be a concave
modulus of continuity of $\tilde\phi$ on $[-M,M]$, $\Pi$ be its
inverse, and $\overline\Pi(r)=r\,\Pi(r)$. Let $\overline\pi$ be
the inverse of $\overline\Pi$. Note that $\overline \pi$ is
concave, continuous, $\overline\pi(0)=0$. Denote
$v=b(u_{n})(t+\Del)$ and $y=b(u_{n})(t)$. We have
\begin{equation*}
\begin{split}
\iint_{Q_{T-\Del}} |\tilde\phi(v)-\tilde\phi(y)|
& =\iint_{Q_{T-\Del}} \overline\pi
\Biggl(\overline\Pi(|\tilde\phi(v)-\tilde\phi(y)|) \Biggr)
\\ & \leq |{Q_{T-\Del}}|\,\overline\pi\Biggl(\frac{1}{|{Q_{T-\Del}}|}
\iint_{Q_{T-\Del}}\overline\Pi(|\tilde\phi(v)-\tilde\phi(y)|) \Biggr).
\end{split}
\end{equation*}
Since $|\tilde\phi(v)-\tilde\phi(y)|\leq \pi(|v-y|)$, we have
$\Pi(|\tilde\phi(v)-\tilde\phi(y)|)\leq |v-y|$ and
\begin{equation*}
\begin{array}{l}
\dsp \overline\Pi(|\tilde\phi(v)-\tilde\phi(y)|) =
\Pi(|\tilde\phi(v)-\tilde\phi(y)|)|\tilde\phi(v)-\tilde\phi(y)|\leq
|v-y|\,|\tilde\phi(v)-\tilde\phi(y)|\\[5pt]
\dsp \qquad\qquad\qquad\quad\; \equiv
|b(u_{n})(t+\Del)-b(u_{n})(t)|\;|\phi(u_{n})(t+\Del)-\phi(u_{n})(t)|.
\end{array}
\end{equation*}
Therefore the estimate \eqref{Transl-base} (with $b_n,\phi_n$ and
$d\,\Del$ replaced by $b,\phi$ and $d\,\Del + r_n$,
respectively) implies
\begin{equation*}
\begin{array}{l}
\dsp \iint_{Q_{T-\Del}}
|\phi(u_{n})(t\!+\!\Del)\!-\!\phi(u_{n})(t)|\\[10pt]
\dsp\qquad \quad \leq
|{Q_{T-\Del}}|\,\overline\pi\Biggl(\frac{1}{|{Q_{T-\Del}}|}
\iint_{Q_{T-\Del}}|u_{n}(t\!+\!\Del)-u_{n}(t)|\,|\phi(u_{n})(t\!+\!\Del)-\phi(u_{n})(t)|
\Biggr)\\[10pt]
\dsp\qquad\quad
=|{Q_{T-\Del}}|\,\overline\pi\Biggl(\frac{1}{|{Q_{T-\Del}}|}
\Del \Biggr)\;\leq\; d \;\overline\pi(\,d\,\Del+r_n\,),
\end{array}
\end{equation*}
and finally, replacing $\phi$ with $\phi_n$ we get
\begin{equation}\label{eq:bound-for-n-large}
\iint_{Q_{T-\Del}}
|\phi_n(u_{n}(t\!+\!\Del,x))\!-\!\phi_n(u_{n}(t,x))|\;\leq\; d
\;\overline\pi(\,d\,\Del+r_n\,) +r_n.
\end{equation}
Now using the fact that $r_n\to 0$ as $n\to\infty$ and the fact
that for all fixed $n\in\N$, the left-hand side of
\eqref{eq:bound-for-n-large} tends to zero as $\Del\to 0$, we
deduce the claim of (iii).
\end{proof}

\section{Proof of the general continuous dependence property}
\label{sec:ContDependence}

In this section, we prove Theorem~\ref{th:ContDep}. First notice
that the uniform estimates of Lemma~\ref{lem:apriori-estimates}
and Lemma~\ref{SmallSetTruncationLemma} apply. It follows that
there exists a (not relabelled) subsequence $(u_{n})_n$
such that\\[3pt]
$\bullet$ $w_{n}:=\phi_{n}(u_{n})$  converges strongly in
$L^1(Q_T)$ and a.e.~on $Q_T$ to some function $w$;\\
$\bullet$  $w_{n}$ converges weakly in $L^p(0,T;W^{1,p}_0(\Om))$;\\
$\bullet$ $\chi_n:={\mathfrak a}(u_n,\grad w_{n})$ converges weakly in $L^{p'}(Q_T)$ to some limit $\chi$;\\
$\bullet$ $u_{n}$ converges to $\mu:Q_T\times(0,1)\longrightarrow
\R$ in the sense of the nonlinear weak-$\star$ convergence
\eqref{nonlinweakstar}.

Denote $u(\cdot)=\int_0^1 \mu(\cdot,\al)\,d\al$. Thanks to the
uniform $L^\infty$ bound on $u_n$ and to the uniform convergence
of $\phi_{n}$ to $\phi$ on compact subsets of $\R$, we can
identify the limit of $w_{n}(\cdot)$ with
$\int_0^1\phi(\mu(\al,\cdot))\,d\al$; moreover,
$\phi(\mu(\alpha,\cdot))$ is independent of $\alpha\in (0,1)$,
because the convergence of $\phi(u_n)$ to $w$ is actually strong.
Thus $w,\phi(u)$ and $\phi(\mu(\alpha))$ coincide. We also
identify the limit of $\grad w_{n}$ with $\grad w$, because the
two functions coincide as elements of $\mathcal D'$.

The following lemma permits to deduce strong convergence of $u_n$
to $u$ on the set $[u\notin E]$ (recall that $E$ is assumed to be
closed).
\begin{lem}\label{lem:PointwiseConvI}
Let $\phi_n(\cdot)$ be a sequence of  continuous non-decreasing
functions converging pointwise to a continuous function
$\phi(\cdot)$. Assume $\phi(\cdot)$ is strictly increasing on
$\R\setminus E$. Let $I$ be an open interval contained in
$\R\setminus E$, and $\phi_n(u_n)\to \phi(u)$ a.e.~on $Q$. Then
$u_n\to u$ a.e.~on $[u\in I]$.
\end{lem}
\begin{proof}
 Let $I=(a,b)$ and choose $I'=(a',b')$ with $a<a'<b'<b$. Introduce
 $\delta>0$ by
$$\delta:=\min\{\phi(a')-\phi(a),\phi(b)-\phi(b')\}.$$

 Notice that by the Dini theorem, the convergence
of $\phi_n(\cdot)$ to $\phi(\cdot)$ is uniform on all compact
subset of $\R$. Thus for all $\eps>0$, there exists $N\in\N$ such
that $\|\phi_n-\phi\|_{C([a,b])}< \eps/2$. With $\eps=\delta$, it
follows that for all $n$ sufficiently large,
$\phi_n(b)-\phi(b')=\phi_n(b)-\phi(b)+\phi(b)-\phi(b')>-\delta/2+\delta=\delta/2$,
and similarly, $\phi(a')-\phi_n(a)>\delta/2$. By the monotonicity
of $\phi_n(\cdot),\phi(\cdot)$,
$$
\max_{u\in I',\, z\notin I} |\phi_n(z)-\phi(u)|
=\max\{\phi_n(b)-\phi(b'),\phi(a')-\phi_n(a)\}>\delta/2.
$$
Hence if $u(t,x)\in I'$ and if $|\phi_n(u_n(t,x))-\phi(u(t,x))|<
\delta/2$, we have $u_n(t,x)\in I$. Since for all $\eps>0$ there
exists $N\in\N$ such that $|\phi_n(u_n(t,x))-\phi(u(t,x))|\leq
\eps/2$, we have in particular that for a.e.~$(t,x)\in [u\in I']$,
one has $u_n\in I$ for all $n$ large enough.

Thus for all $\eps<\delta$, for a.e.~point of $[u\in I']$ there
exists $N\in \N$ such that at this point one has
\begin{align*}
|\phi(u_n)-\phi(u)| & \leq
|\phi(u_n)-\phi_n(u_n)|+|\phi_n(u_n)-\phi(u)|
\\ & \leq \|\phi_n-\phi\|_{C([a,b])}+ \eps/2\leq \eps.
\end{align*}
Therefore $\phi(u_n)$ converges to $\phi(u)$ a.e.~on $[u\in I']$.
Since $\phi(\cdot)$ is continuously invertible on $I$, one also
has $u_n\to u$ a.e.~on $[u\in I']$. Since $I$ is open and
$I'\Subset I$ is arbitrary, the claim of the lemma follows.
\end{proof}

Now we start to identify $\chi$ with $\mathfrak a(u,\grad w)$.
First, according to Remark~\ref{rem:NeglectedSet}, $\mathfrak
a(u,\grad w)=0$ a.e.~on the set $[u\in E]= [w\in G]$. Using
Lemma~\ref{SmallSetTruncationLemma}, we now deduce that also
$\chi=0$ on this set.
\begin{lem}\label{lem:IdentOnE}
Let $\chi$ be the weak $L^{p'}(Q_T)$ limit of the sequence
$$
\chi_n=\mathfrak a_n(u_n,\grad \phi_n(u_n)),
$$
and let $w_n=\phi_n(u_n)$ converge to $w=\phi(u)$ a.e..
Then $\chi=0$ a.e.~on $[u\in E]$.
Moreover, $\chi_n$ converges strongly
to zero in $L^1([u\in E])$.
\end{lem}
\begin{proof}
Since $\char_{[u\in E]}\in L^\infty(Q_T)\subset L^{p}(Q_T)$, by
the definition of the weak convergence, the function
$\chi\char_{[u\in E]}\equiv\chi\char_{[w\in G]} $ is the weak
$(L^{p'}(Q_T))^N$ limit of $\chi_n\char_{[w\in G]}$. For all
$\eps>0$, choose an open set $G_\eps\supset G$ of measure less
than $\eps$. Because $w_n\to w$ a.e.~and $G_\eps$ is an open
neighbourhood of $G$, we have
$$
[u\in E]=[w\in G]=
R\cup\Bigl(\cup_{N\in \N} [w\in G,\, w_n\in G_\eps\,
\forall n\geq N] \Bigr),
$$
where $\meas(R)=0$. Since
$$
\Biggl([w\in G,\, w_n\in G_\eps\, \forall n\geq N]\Biggr)_{N\in \N}
$$
is an increasing collection of sets, the corresponding
sequence of measures converges to
$\meas([w\in G])$ as $N\to \infty$,
by the continuity of the Lebesgue measure.
Because
$$
\meas ([w\in G]) \geq \meas ([w\in G,\, w_N\in G_\eps]) \geq
\meas([w\in G,\, w_n\in G_\eps\, \forall n\geq N]),
$$
we conclude that $\meas([w\in G,\,w_n\notin G_\eps])$ tends to
zero as $n\to\infty$. Now by the H\"older inequality,
\begin{align*}
\|\chi_n\char_{[w\in G]}\|_{L^1(Q)}
& =\iint_{[w\in G]}
|\mathfrak a_n(u_n,\grad w_n)|
\\ & \leq \iint_{[w_n\in G_\eps]}
|a_n(u_n,\grad w_n)|
\\ & \qquad
+ \Biggl(\iint_{Q_T} |a_n(u_n,\grad w_n)|^{p'}\Biggr)^\frac{1}{p'}
\\ & \qquad\qquad \times
\Biggl(\meas([w\in G,\,w_n\notin G_\eps]) \Biggr)^\frac{1}{p}.
\end{align*}
For all fixed $\eps>0$, the last term tends to zero as $n\to
\infty$, thanks to the boundedness of the sequence $\mathfrak
a_n(u_n,\grad w_n)$ in $L^{p'}(Q_T)$. The first term in the
right-hand side converges to zero as $\eps\to 0$ uniformly in $n$;
indeed, by $(H_9)$ and by Lemma~\ref{SmallSetTruncationLemma}, it
is majorated by $\frac{1}{c}\meas(G_\eps)\leq \eps/c$.

It follows that $\chi \char_{[u\in E]}=0$ a.e.~on $Q_T$, which
ends the proof.
\end{proof}

Now we use the Minty-Browder argument to identify $\chi$
with $\mathfrak a(u,\grad w)$ on the sets $[u\in I]$, for all open
$I\subset \R\setminus E$. A crucial role is played by Lemma~\ref{lem:WeakWeakOk}, which
permits us to pass to the limit in the product
$$
{\mathfrak f}_{n}(u_{n})\cdot \grad w_n
$$
of two weakly converging sequences.

We proceed in the classical way, but use the test functions
$T_J(w)$ (at the limit) and $T_J(w_n)$ (before the passage to the
limit), where $J=\phi(I)$ and
$$
T_J(z):=\dsp\int_0^z \char_{J}(s)\,ds
$$
is the truncation function that localizes the
values of the solution to the interval $I$. Notice that we can
assume that $\grad T_J(w_n)$ converges to $\grad T_J(w)$ weakly in
$(L^p(Q_T))^N$. Indeed, the sequence $\Bigl( T_J(w_n) \Bigr)_n$ is
bounded in $L^p(0,T;W^{1,p}_0(\Om))$, hence (up to a subsequence)
it converges to a limit; this limit is identified with $T_J(w)$
since $w_n$ converges to $w$ in $L^1(Q_T)$, and $T_J$ is Lipschitz
continuous. Now let us give the details.

 We first pass to the limit into the
weak formulation of $(P_n)$ (see
Remark~\ref{rem:weak-AL-formulation}), getting
\begin{equation}\label{eq:LimitEqTheor}
\left\{\begin{array}{l} \dsp\pt \biggl(\int_0^1
\!\!b(\mu(\al))\,d\al\biggr)+\div\! \biggl(\int_0^1 \!\!{\mathfrak
f}(\mu(\al))\,d\al\biggr)+\int_0^1 \!\!\psi(\mu(\al))\,d\al= \div
\chi\; +\; f\\[5pt]
\hspace*{50mm}\quad \text{in $L^{p'}(0,T;W^{-1,p'}(\Om))+L^1(Q_T)$},\\
\dsp\int_0^1 \!\!b(\mu(\al))\,d\al|_{t=0}=b(u_0)
\end{array}\right.
\end{equation}
in the same sense as in Remark~\ref{rem:weak-AL-formulation}. Take
$\zeta\in \mathcal D([0,T))$. Because $w=\phi(u)\equiv
\phi(\mu(\al))$ for a.e.~$\alpha\in (0,1)$, using
Lemma~\ref{lem:IBP-measures} we have
\begin{equation}\label{eq:Ch-rule-mu}
\begin{split}
& \int_0^T\!\! \left \langle\pt \biggl(\int_0^1
\!\!b(\mu(\al))\,d\al\biggr)\,,\, T_J(\phi(u))\zeta\right \rangle
\\ & \qquad = -\iint_{Q_T} \!\int_0^1 \!\!B_J(\mu(\al))\,d\al\,
\pt \zeta-\int_\Om \!B_J(u_0) \zeta(0)
\end{split}
\end{equation}
with $\dsp B_J(z):=\int_0^z T_J(\phi(s))\,db(s)$. Define similarly
$\dsp B^n_J(z):=\int_0^z T_J(\phi_n(s))\,db_n(s)$.   By the
standard chain rule lemma of
\cite{AltLuckhaus,Otto:L1_Contr,CarrilloWittbold} we get
\begin{equation}\label{eq:Ch-rule-u-n}
 \int_0^T\!\! \left \langle \pt
b_n(u_n)\,,\, T_J(\phi_n(u_n))\zeta\right
\rangle = -\iint_{Q_T} \!\!B^n_J(u_n)\, \pt \zeta
-\int_\Om \!B^n_J(u^n_0) \zeta(0).
\end{equation}
One shows easily that $B^n_J$ converges to $B_J$ uniformly on
compact subsets of $\R$, because of \eqref{eq:converg-data}. In
particular, $B^n_J(u_0)$ converge to $B_J(u_0)$ in $L^1(\Om)$.
Moreover, the nonlinear weak-$\star$ convergence of $u_n$ to $\mu$
yields
$$
\lim_{n\to+\infty} \iint_{Q_T} \!\!B^n_J(u_n)\; \pt \zeta=
\iint_{Q_T} \!\int_0^1 \!\!B_J(\mu(\al))\,d\al\; \pt \zeta.
$$
Thus the left-hand side of \eqref{eq:Ch-rule-mu} coincides with
the ``$n\to+\infty$ limit" of the left-hand side of \eqref{eq:Ch-rule-u-n}.

Similarly, the nonlinear weak-$\star$ convergence of $u_n$ to $\mu$
permits to conclude that
\begin{align*}
\iint_{Q_T}\! \int_0^1\! \psi(\mu(\al))\,d\al\;
T_J(\phi(u))\zeta
& =\iint_{Q_T}\! \int_0^1\! \psi(\mu(\al))
T_J(\phi(\mu(\al)))\,d\al\,\zeta
\\ & =\lim_{n\to+\infty}\iint_{Q_T}
\!\!\psi_n(u_n) T_J(\phi_n(u_n))\,\zeta.
\end{align*}
Now let us take the test functions $T_J(w_{n})\zeta$
and $T_J(w)\zeta$ in the weak formulation of
$(P_n)$  and in \eqref{eq:LimitEqTheor}, respectively.
Without loss of restriction, we can assume
that $t=T$ is a Lebesgue point of $\dsp \int_0^1 \!\!B_J(\mu(\al))\,d\al$.
Using the above calculations
and Lemma~\ref{lem:WeakWeakOk}, then letting
$\zeta(t)$ converge to $\char_{[0,T)}(t)$ we deduce the equality
\begin{equation}\label{eq:MintyIneqTheor}
\iint_{Q_T} \chi\cdot \grad T_J(w) = 
\lim_{{n}\to \infty} \iint_{Q_T} {\mathfrak a}_n(u_n,\grad w_{n})
\cdot \grad T_J(w_{n}).
\end{equation}
\begin{lem}\label{lem:MintyTruncated}
Assume that for all $\xi\in \R^N$, $\mathfrak a_n(\cdot,\xi)$
converge to $\mathfrak a(\cdot,\xi)$ uniformly on compact subsets
of an open interval $I\subset \R\setminus E$.
 With the notation and assumptions above, one has
$\chi={\mathfrak a}(u,\grad w)$ a.e.~on $[u\in I]=[w\in J]$.
\end{lem}
\begin{proof}
Take an arbitrary function $\tilde \zeta \in (L^{p}(Q_T))^N$ such
that $\tilde \zeta=0$ a.e.~on $[w\notin J]$. Take $\lambda\in \R$.
Set $\zeta=\grad T_J(w)+\lambda \tilde \zeta$; we have $\zeta=0$
a.e.~on $[w\notin J]$. By the classical Minty-Browder argument,
considering $\pm\lambda\downarrow 0$ one concludes that
$\chi=\mathfrak a(u,\grad w)$ a.e.~on $[w\in J]$, provided the
following relations can be justified:
\begin{equation}\label{eq:TheMintyOnJ}
\begin{split}
\iint_{Q_T} \chi\cdot (\grad T_J(w)-\zeta)
& \geq \liminf_{n\to\infty}
\iint_{Q_T} \mathfrak a_n(u_n,\grad w_n)
\cdot (\grad T_J(w_n)-\zeta)\\
& \geq \liminf_{n\to\infty}
\iint_{Q_T} \mathfrak a_n(u_n,\zeta)
\cdot (\grad T_J(w_n)-\zeta)
\\ & =\liminf_{n\to\infty} \iint_{Q_T} \mathfrak a(u,\zeta) \cdot
(\grad T_J(w_n)-\zeta)\\
& =\iint_{Q_T} \mathfrak a(u,\zeta)
\cdot (\grad T_J(w)-\zeta).
\end{split}
\end{equation}
Now we justify \eqref{eq:TheMintyOnJ}. Because
\eqref{eq:MintyIneqTheor} holds and $\grad T_J(w_n)$ converges to
$\grad T_J(w)$ weakly in $(L^p(Q_T))^N$, the first inequality and
the last equality in \eqref{eq:TheMintyOnJ} are clear.

The second inequality comes from the monotonicity
of $\mathfrak a_n(u_n,\cdot)$. Indeed, by the choice of $\zeta$,
\begin{align*}
&\iint_{[w\notin J]}\!\!\! \mathfrak a_n(u_n,\grad w_n)
\cdot(\grad T_J(w_n)-\zeta)
\\ & =\iint_{[w\notin J,\, w_n\in J]}\!\!\!\!\!\!  \mathfrak a_n(u_n,\grad w_n)
\cdot\grad w_n
\\ & \ge 0 = \iint_{[w\notin J]}\!\!\! \mathfrak a_n(u_n,\zeta)
\cdot (\grad T_J(w_n)-\zeta),
\end{align*}
because $\mathfrak a_n(\cdot,0)\equiv 0$.
Further, since $J$ is open and $w_n\to w$ a.e.~on $Q_T$, as in the proof of
Lemma~\ref{lem:IdentOnE} we have $\meas([w\in J, w_n\notin J])\to
0$ as $n\to \infty$. On $[w\in J, w_n\in J]$ one has $\mathfrak
a_n(u_n,\grad w_n)=\mathfrak a_n(u_n,\grad T_J(w_n))$, and $(H_8)$
applies. On $[w\in J, w_n\notin J]$ the terms with $\grad
T_J(w_n)$ are zero; finally, as in the proof of
Lemma~\ref{lem:IdentOnE}, the integrals
$$
\iint_{[w\in J, w_n\notin J]} \mathfrak{a}_n(u_n,\grad
w_n)\cdot\zeta \quad \text{and} \quad \iint_{[w\in J, w_n\notin
J]} \mathfrak{a}_n(u_n,\zeta)\cdot\zeta
$$
both tend to zero, by the equi-integrability argument.

It remains to justify the last but one equality in
\eqref{eq:TheMintyOnJ}. Thanks to Lemma~\ref{lem:PointwiseConvI},
we have $u_n\to u$ a.e.~on $[w\in J]=[u\in I]$. In particular,
a.e.~on $[u\in I]$ one has $u_n\in I$ for sufficiently large $n$.
Using the locally uniform on $I$ convergence of $\mathfrak
a_n(\cdot,\xi)$ to $\mathfrak a(\cdot,\xi)$, we conclude that
$\mathfrak a_n(u_n,\zeta)$ converges to $\mathfrak a(u,\zeta)$
a.e.~on $[u\in I]$. By $(H_{10})$ and the Lebesgue dominated
convergence theorem, \eqref{eq:TheMintyOnJ} follows.
\end{proof}

Finally, notice that  thanks to \ref{def:entropy1p} and
Remark~\ref{rem:NeglectedSet}, we have $\mathfrak a(\mu,\grad w)\equiv \mathfrak a(u,\grad w)$.
The identification of $\chi$ thus being complete, from \eqref{eq:LimitEqTheor} we readily conclude that
$(\mu,w)$ verifies \ref{def:entropyWeakp}.

Now we can pass to the limit in the entropy
inequalities corresponding to $(P_n)$ and deduce
\ref{def:entropy2p}. Because regular boundary entropy pairs
$(\eta_{c,\eps}^\pm,\mathfrak{q}_{c,\eps}^\pm)$ can be
approximated by convex combinations of semi-Kruzhkov pairs, we
have the analogue of \ref{def:entropy2} for $u_n$ with
$(\eta_{c}^\pm,\mathfrak{q}_{c}^\pm)$ replaced by
$(\eta_{c,\eps}^\pm,\mathfrak{q}_{c,\eps}^{n,\pm})$ (with
$\mathfrak{q}_{c,\eps}^{n,\pm}$ converging to
$\mathfrak{q}_{c,\eps}^{\pm}$ uniformly on compact subsets of $\R$).

Consider the third term in  \ref{def:entropy2}. We have
\begin{equation}\label{eq:pass-to-the-lim-entropic}
\begin{split}
& \lim_{n\to\infty} \iint_{Q_T} (\eta_{c,\eps}^\pm)'(u_n)
\mathfrak a_n(u_n,\Grad w_n)  \cdot \grad \xi
\\ & \quad =\lim_{n\to\infty} \Biggl(\iint_{[w\in G]}+\iint_{[w\notin G]} \Biggr)
(\eta_{c,\eps}^\pm)'(u_n)\grad \xi\cdot \mathfrak{a}_n(u_n,\Grad w_n).
\end{split}
\end{equation}
By Lemma~\ref{lem:IdentOnE} and because
$(\eta_{c,\eps}^\pm)'(u_n)\grad \xi$ are bounded uniformly in $n$,
the first term converges to zero; also notice that
$$
0=\iint_{[w\in G]} (\eta_{c,\eps}^\pm)'(u)\grad \xi\cdot
\mathfrak{a}(u,\Grad w)
$$
because $\mathfrak{a}(u,\Grad w)
=\mathfrak{a}(u,0)=0$ a.e.~on $[w\in G]$. By
Lemma~\ref{lem:PointwiseConvI}, we have $u_n\to u$ a.e.~on
$[w\notin G]$; by the dominated convergence theorem, we infer that
$(\eta_{c,\eps}^\pm)'(u_n)\grad \xi$ converges to
$(\eta_{c,\eps}^\pm)'(u)\grad \xi$ strongly in $L^p([w\notin G])$.
Because also $\chi_n=a_n(u_n,\Grad w_n)$ converges to
$\chi=a(u,\Grad w)$ weakly in $L^{p'}([w\notin G])$, the second
term in the right-hand side of \eqref{eq:pass-to-the-lim-entropic}
converges to
$$
\int_{[w\notin G]} (\eta_{c,\eps}^\pm)'(u)\grad \xi
\cdot \mathfrak{a}(u,\Grad w).
$$

The passage to the limit  as $\eps\to 0$ in the other terms in
\ref{def:entropy2} is straightforward, using the uniform
boundedness of $(u_n)_n$ and the nonlinear weak-$*$ convergence
property \eqref{nonlinweakstar}. At the limit, we conclude that
\ref{def:entropy2p} also holds. Thus
 $(\mu,w)$ is an entropy process solution of $(P)$.

Now by the result of Theorem~\ref{th:EPSunique}(i),  $(\mu,w)$
gives rise to an entropy solution
$$
u=\dsp \int_0^1 \mu(\al)\,d\al
$$
of $(P)$. By \eqref{nonlinweakstar} (with $F=\text{Id})$, we
conclude that $(u_n)_n$ possesses a subsequence that converges in
$L^\infty(Q_T)$ weak-$\star$ to $u$; we have already shown that the
corresponding subsequence $\phi_n(u_n)$
converges to $\phi(u)$ in $L^1(Q_T)$.

Moreover, $b(\mu(\al))$ and $\psi(\mu(\al))$ are in fact
independent of $\alpha$. By Theorem~\ref{th:EPSunique}(iii), we
also have the uniqueness of $b(u)$ and $\psi(u)$ such that $u$ is
an entropy solution of $(P)$. By the well-known result of the
nonlinear weak-$\star$ convergence (see e.g., Hungerb\"uhler
\cite{Hungerbuhler}), it follows that the whole sequences
$(b_n(u_n))_n$ and $(\psi_n(u_n))_n$ converge to $b(u)$ and
$\psi(u)$, respectively, in measure on $Q_T$ and in $L^1(Q_T)$.

This ends the proof of Theorem~\ref{th:ContDep}.

\begin{rem}\label{rem:bounded-b}
In the case  assumption $(H_5)$ is
dropped, in order to deduce that $u$ is
an entropy process solution, along with the
assumption $(H_{str})$ one needs a growth
restriction on $\tilde{\mathfrak f}$ of the following kind: there exists a function
$M\in C(\R^+;\R^+)$ and a function $\mathcal L$
with $\lim\limits_{z\to +\infty} {\mathcal L}(z)/z=0$ such that
\begin{align*}
|\tilde{\mathfrak
f}(b(r),\phi(r),\psi(r))|\leq M(|b(r)|)\;
{\mathcal L}\left(|\phi(r)|^{p}\! +\!\!\int_0^r\! \phi(s)\,db(s)+\psi(r)\phi(r)\right)
\end{align*}
for all $r\in\R$, and the same assumption with $|\tilde{\mathfrak
f}(b(r),\phi(r),\psi(r))|$ replaced with $|b(r)|$ and with
$|\psi(r)|$. Indeed, these inequalities make it possible to use
the nonlinear weak-$\star$ convergence framework of Ball
\cite{Ball} and Hungerb\"uhler \cite{Hungerbuhler} without the
uniform $L^\infty$ bound on $(u_n)_n$.
\end{rem}

\section{Proof of Theorem~\ref{th:EPS-ESexists}}
\label{sec:General-b}

In this section, we prove Theorem~\ref{th:EPS-ESexists}. The
uniqueness claim was shown in Theorem~\ref{th:EPSunique}; also
notice that the continuous dependence result under the structure
assumption $(H_{str})$ was shown in Theorem~\ref{th:ContDep}. Let
us first prove the existence claim.

(i) First, consider the case where assumption
$(H_{str})$ is fulfilled. Consider an approximation of $(P)$ by
regular problems ($P_{n}$) with data
$(b_n,\psi,\phi_n,\mathfrak{a},\tilde{\mathfrak{f}};u_0,f)$ such
that the assumptions of Theorem~\ref{th:ContDep} are fulfilled,
and $b_n,[b_n]^{-1},{\mathfrak f}_{n},\phi_{n},[\phi_{n}]^{-1}$
are Lipschitz continuous on $\R$.
 Using  classical methods (cf. Alt and Luckhaus
\cite{AltLuckhaus}, Lions \cite{Lions}), one shows that there
exists a weak solution $u_{n}\in L^p(0,T;W^{1,p}_0(\Om))$ to the
problem $(P_{n})$ in the sense
$$
\pt b_n(u_{n})+\div \tilde{\mathfrak f}_{n}(u_{n})+\psi(u_n)
= \div \mathfrak{a}(u_n,\grad
\phi_{n}(u_{n})) + f
$$
in $L^{p'}(0,T;W^{-1,p'}(\Om))+L^1(Q_T)$, with initial data
$$
b_n(u_{n})|_{t=0}=b_n(u_0).
$$
In addition, $u_{n}$ verifies the entropy formulation
\ref{def:entropy2}, obtained along the lines of Carrillo
\cite{Carrillo}. By Theorem~\ref{th:ContDep}, we conclude that
there exists an entropy solution of $(P)$.

To prove existence without the structure condition $(H_{str})$, we
use the particular multi-step approximation approach of Ammar and
Wittbold \cite{AmmarWittbold} (see also Ammar and Redwane
\cite{AmmarRedwane}). We replace $b$ by $b_{k}:=b+\frac 1k
\text{Id}$ and $\psi$, by $\psi_{m,n}:=\psi+\frac 1n \text{Id}^+ +
\frac 1m \text{Id}^-$.
 The result of (i) for the
corresponding problem $(P_{m,n}^k)$ is already proved.

There exists a function $u^k_{m,n}$, constructed by means of the
nonlinear semigroup theory (see e.g., \cite{BCP}), such that
$b_k(u^k_{m,n})$ is the unique integral solution to the abstract
evolution problem associated with $(P_{m,n}^k)$ (here and below,
we refer to  Ammar and Wittbold \cite{AmmarWittbold}, Ammar and
Redwane \cite{AmmarRedwane} for details). One then shows that
$u^k_{m,n}$ coincides with the unique entropy solution of
$(P_{m,n}^k)$, the existence of this entropy solution being
already shown. Further, the whole set $(u^k_{m,n})_{k,m,n}$
verifies the uniform {\it a priori} estimates of Lemmas
\ref{lem:apriori-estimates}, \ref{SmallSetTruncationLemma}.

We then pass to the limit in $u^k_{m,n}$ in the following order:
first $k\to +\infty$, then $n\to +\infty$, $m\to +\infty$.

While letting $k\to+\infty$, we use the fact that
$\psi_{m,n}^{-1}$ is Lipschitz continuous. The fundamental
estimates for the semigroup solutions permit to show that
$\psi_{m,n}(u_{m,n})$ are uniformly continuous on $(0,T)$  with
values in $L^1(\Om)$; thus we get the strong precompactness of
$(u^k_{m,n})_k$ in $L^1(Q_T)$. Thus, up to a subsequence,
$u^k_{m,n}$ converge to $u_{m,n}$ which is an entropy solution of
problem $(P_{m,n})$ corresponding to the data
$(b,\psi_{m,n},\phi,\mathfrak{a},\tilde{\mathfrak{f}};u_0,f)$.

Finally, we use the inequalities $u_{m+1,n}\leq u_{m,n} \leq
u_{m,n+1}$ which follow readily form the comparison principle of
Theorem~\ref{th:EPSunique}(ii). The monotonicity argument yields
the strong convergence of $u_{m,n}$. Now the whole scheme of the
proof of Theorem~\ref{th:ContDep} applies with considerable
simplifications, because no nonlinear weak-$\star$ convergence arguments
are not needed. Passing to the limit in $u_{m,n}$, we conclude
that the limit $u$ is an entropy solution of the original problem
$(P)$. This ends the existence proof.

(ii) Existence for the limit data $(u_0,f)$ is now shown
and we can apply Theorem~\ref{th:EPSunique}(ii). Then we deduce
the $L^1(Q_T)$ convergence of $b(u_n),\psi(u_n)$ to
$b(u),\psi(u)$, respectively. The convergence of $\phi(u_n)$ to
$\phi(u)$ in $L^1(Q_T)$ follows, by hypothesis $(H'_{str})$ and
because our assumptions imply the uniform $L^\infty(Q_T)$ bound on
$u_n$.

 The remaining claim of the strong $(L^p(Q_T))^N$
convergence of $\grad \phi(u_n)$ is a rather standard part of the
Minty-Browder trick.
 The argument is based upon the
proof of Theorem~\ref{th:ContDep}. Because we already have the
strong compactness of $(\phi(u_n))_n$, we can bypass the
hypothesis $(H_{str})$ and deduce that the $L^\infty$ weak-$\star$ limit
$\hat u$ of $u_n$ is an entropy solution of $(P)$
 with the data $(u_0,f)$. In particular, the $(L^{p'}(Q_T))^N$ weak limit $\chi$ of
$\mathfrak{a}(u_n,\grad\phi(u_n))$ is  equal to $\mathfrak{a}(\hat
u,\grad\phi(\hat u))$. Now notice that by
Theorem~\ref{th:EPSunique}(iii), under the structure condition
$(H'_{str})$ we also have the uniqueness of $\phi(u)$ such that
$u$ is an entropy solution of $(P)$; moreover, by
Remark~\ref{rem:NeglectedSet}, we also have the uniqueness of
$\mathfrak{a}(u,\grad\phi(u))$ such that $u$ is an entropy
solution of $(P)$. Thus  $\chi$ coincides with $\mathfrak{a}(
u,\grad\phi( u))$, so that \eqref{eq:MintyIneqTheor} now reads as
\begin{equation}\label{eq:MintyIneqTheor-2}
\iint_{Q_T} \mathfrak{a}(u,\!\grad\phi(u))\cdot\!\grad\phi(u) =
\lim_{{n}\to \infty} \iint_{Q_T} {\mathfrak a}_n(u_n,\!\grad
w_{n}) \cdot \!\grad \phi(u_n).
\end{equation}
It follows by the weak convergences of $\grad \phi(u_n)$ and of
$\mathfrak{a}(u_n,\grad\phi(u_n))$ that
\begin{equation}\label{eq:MintyIneqTheor-3}
 \lim_{{n}\to \infty}
\iint_{Q_T} \Bigl( \mathfrak{a}(u,\!\grad\phi(u))-{\mathfrak
a}(u_n,\!\grad \phi(u_n))\Bigr) \cdot \Bigl( \grad\phi(u)-\grad
\phi(u_n)\Bigr)=0.
\end{equation}  Notice that
a.e.~on the set $[u\in E]$,
  $\grad\phi(u_n)$ converges to $0=\grad \phi(u)$, by Lemma~\ref{lem:IdentOnE} and by the coercivity assumption
$(H_9)$. On the set $[u\notin E]$, we can use
Lemma~\ref{lem:PointwiseConvI} to replace
$\mathfrak{a}(u,\!\grad\phi(u))$ with
$\mathfrak{a}(u_n,\!\grad\phi(u))$ in the above formula
\eqref{eq:MintyIneqTheor-3}.
 Using the uniform  monotonicity assumption $(H'_8)$, we can now conclude that the convergence of $\grad\phi(u_n)$ to
$\grad \phi(u)$ holds a.e.~on $Q_T$.

Separating again the sets $[u\in E]$ and $[u\notin E]$, we deduce
that the sequence of nonnegative functions
$\mathfrak{a}(u_n,\!\grad\phi(u_n))\cdot\!\grad\phi(u_n)$
converges to $\mathfrak{a}(u,\!\grad\phi(u))\cdot\!\grad\phi(u)$
a.e.~on $Q_T$. Together with \eqref{eq:MintyIneqTheor-2}, this
implies that
$\mathfrak{a}(u_n,\!\grad\phi(u_n))\cdot\!\grad\phi(u_n)$ also
converge to $\mathfrak{a}(u,\!\grad\phi(u))\cdot\!\grad\phi(u)$ in
$L^1(Q_T)$; in particular, they are equi-integrable on $Q_T$. The
coercivity assumption $(H_9)$ now implies the equi-integrability
on $Q_T$ of the functions $|\grad \phi(u_n)|^p$. Combining this
argument with the a.e.~convergence of $\grad \phi(u_n)$, we deduce
our claim from the Vitali theorem. The $(L^{p'}(Q_T))^N$
convergence of $\mathfrak{a}(u_n,\grad\phi(u_n))$  to
   $\mathfrak{a}(u,\grad\phi(u))$ follows in the same way.

\section{Well-posedness for the doubly nonlinear elliptic problem}
\label{sec:Extensions}
We first notice that the well-posedness result for the degenerate
elliptic  problem
$$
 \leqno (S) \qquad  \left\{\begin{array}{l}
      \displaystyle  \psi(u)+ \div { \tilde{\mathfrak
      f}}(\psi(u),\phi(u))-
      \div { \mathfrak a}(u,\grad \phi(u))=s \quad \text{in $\Om$},
      \\u=0 \qquad \text{on $\ptl\Om$}
   \end{array}\right.
$$
$$
\leqno (H'_5) \qquad \text{$s\in L^\infty(\Om)$
and $\psi(\R)=\R$ }
$$
 follows from
Theorem~\ref{th:EPS-ESexists}, upon setting $b\equiv 0$,
$f(t,\cdot)\equiv s(\cdot)$ and arbitrarily prescribing $u_0$. The
definition of an entropy solution of $(S)$ can also be formally
obtained from Definition~\ref{def:Entropy-Solution}.

Let us notice that the analogue of the general continuous
dependence property of Theorem~\ref{th:ContDep} holds without any
additional structure condition:
\begin{theo}\label{th:ContDep-Stat}
 Let
$(\psi_n,\phi_n,\mathfrak{a}_n,\tilde{\mathfrak{f}}_n;s)$,
$n\in\N$, be a sequence converging to
$(\psi,\phi,\mathfrak{a},\tilde{\mathfrak{f}};s)$ in the following
sense:
\begin{equation*}
\begin{array}{ll}
\cdot & \psi_n,\phi_n \text{~converge pointwise
to~} \psi,\phi \text{~respectively};\\
\cdot & \tilde{\mathfrak{f}}_n,\mathfrak{a}_n \text{~converge to~}
\tilde{\mathfrak{f}},\mathfrak{a}, \text{~respectively, uniformly on compacts};\\
 \cdot & s_n \text{ converges to } s
 \text{ in } L^1(\Om).
\end{array}
\end{equation*}
Assume that $(\psi,\phi,\mathfrak{a},\tilde{\mathfrak{f}};s)$ and
$(\psi_n,\phi_n,\mathfrak{a}_n,\tilde{\mathfrak{f}}_n;s_n)$ (for each $n$)
satisfy the hypotheses $(H_1)$, $(H_6)$-$(H_{11})$, and $(H'_5)$,  and, moreover, that the
functions $C$ in $(H_9)$, $(H_{10})$, and $(H_{11})$
as well as the $L^\infty(\Om)$  bound in $(H'_5)$ are independent of $n$.
We denote by $(S_n)$ the analogue of  problem $(S)$ corresponding
to the data and coefficients
$(\psi_n,\phi_n,\mathfrak{a}_n,\tilde{\mathfrak{f}}_n;s_n)$.

Assume that $\phi$ satisfies the technical hypotheses
$(H_2)$,$(H_3)$.

Let $u_n$ be an entropy solution of problem $(S_n)$. Then the functions $u_n$ converge to an
entropy solution $u$ of $(S)$ in $L^\infty(\Om)$ weakly-*, up to a
subsequence. Furthermore, the functions $\phi_n(u_n)$ converge to $\phi(u)$ in
$L^p(\Om)$ up to a subsequence, and the whole
sequence $\psi_n(u_n)$ converges
to $\psi(u)$ in $L^1(\Om)$.
\end{theo}
\noindent The proof of Theorem~\ref{th:ContDep-Stat} is contained
in the one of Theorem~\ref{th:ContDep}, because the $L^p(\Om)$
bound on $\grad w_n$ is sufficient for the strong precompactness
of $w_n$.

\end{document}